\documentclass{amsart}
\usepackage[utf8]{inputenc}
\usepackage{stmaryrd}
\usepackage{latexsym, amsmath, amssymb, amsfonts, amscd, amsthm}
\usepackage{bbm}
\usepackage{epsfig}

\usepackage{amsxtra}
\usepackage[size=tiny, backgroundcolor=white]{todonotes}

\usepackage{enumitem}
\usepackage[T1]{fontenc}
\usepackage{lmodern}
\usepackage{hyperref}
\usepackage[mathscr]{eucal}
\usepackage{scalerel}
\usepackage{pgf,calc,pgffor,xkeyval}
\usepackage[all]{xy}
\usepackage{tikz}
\usepackage{tikz-cd}
\usepackage{xcolor}
\usetikzlibrary{shapes.geometric, arrows}
\usepackage{booktabs} % To thicken table lines
\newtheorem{thm}{Theorem}[section]
\newtheorem{lem}[thm]{Lemma}\newtheorem{lemma}[thm]{Lemma}
\newtheorem{cor}[thm]{Corollary}
\newtheorem{pro}[thm]{Proposition}
\newtheorem{observation}[thm]{Observation}
\newtheorem{example}[thm]{Example}
\newtheorem{remark}[thm]{Remark}
\newtheorem{notation}[thm]{Notation}
\newtheorem{definition}[thm]{Definition}
\newtheorem*{thm*}{Theorem}
\newtheorem*{cor*}{Corollary}

\newcommand{\Cat}{\mathcal C}  
\newcommand{\Mfd}{\mathsf{Mfd}}
\newcommand{\NMfd}{\mathsf{NMfd}}
\newcommand{\String}{\mathsf{String}}
\newcommand{\bt}{\mathsf{t}}

\newcommand{\GpdBibd}{{\bf GpdBibd}}
\newcommand{\R}{\mathbb R}
\newcommand{\Z}{\mathbb Z}
\newcommand{\Real}{\mathbb R}
\newcommand{\huaB}{\mathcal{B}}
\newcommand{\huaF}{\mathcal{F}}
\newcommand{\huaG}{\mathcal{G}}
\newcommand{\huaC}{{\mathcal{C}}}
\newcommand{\huaT}{\mathcal{T}}
\newcommand{\huaM}{\mathcal{M}}
\newcommand{\huaN}{\mathcal{N}}
\newcommand{\g}{\mathfrak g} 
\newcommand{\pt}{\mathbb{\star}}
\newcommand{\Hom}{\mathrm{Hom}}
\newcommand{\cohom}{\mathrm{cohom}}
\newcommand{\alg}{\mathrm{alg}}

\newcommand{\Vectbd}{\mathsf{Vectbd}}
\newcommand{\Aut}{\mathrm{Aut}}
\newcommand{\Diff}{\mathsf{Diff}}
\newcommand{\id}{\mathrm{id}}
\newcommand{\Set}{\mathsf{Set}}

\newcommand*{\Horn}[2]{{\wedge^{#1}_{\!\!\!\!#2}}}
\newcommand{\simp}[1]{{#1}_\bullet}
\newcommand{\Gpd}{{\bf Gpd}}
\newcommand{\bs}{\mathsf{s}}
\newcommand{\pSh}[1]{\mathsf{pSh}({#1})}

\usepackage[final]{changes}

\title{Differentiating $L_\infty$ groupoids}

\author[D. Li]{Du Li}
\address{Du Li, Georg-August-Universität Göttingen,
Institut für Mathematik, Bunsenstr. 3-5, 37073 Göttingen}
\email{liduzg@gmail.com}

\author[L.Ryvkin]{Leonid~Ryvkin}
\address{Leonid Ryvkin, Georg-August-Universität Göttingen,
Institut für Mathematik, Bunsenstr. 3-5, 37073 Göttingen // Institut Camille Jordan, Université Claude Bernard Lyon 1, 
43 boulevard du 11 novembre 1918, 69622 Villeurbanne France}
\email{leonid.ryvkin@math.univ-lyon1.fr}

\author[A. Wessel]{Arne Wessel}
\address{Arne Wessel, Fraunhofer-Institut für Energiewirtschaft und Energiesystemtechnik IEE,
Joseph-Beuys-Straße 8, Kassel}
\email{arne.wessel@iee.fraunhofer.de}

\author[C. Zhu]{Chenchang Zhu}
\address{Chenchang Zhu, Georg-August-Universität Göttingen,
Institut für Mathematik, Bunsenstr. 3-5, 37073 Göttingen}
\email{chenchang.zhu@mathematik.uni-goettingen.de}

\date{\today}

\begin{document}

\begin{abstract}
La différenciation d'un n-groupoïde de Lie via l'analgoue différentio-géométrique des points proches a priori ne donne qu'un pré-faisceau de variétés graduées. Dans cet article, nous démontrons que ce pré-faisceau est représentable par le complexe tangent du n-groupoïde de Lie. En conséquence immédiate, nous obtenons que le complexe tangent d'un tel n-groupoïde porte la structure d'une n-algèbre de Lie. \\
\\
Differentiating a Lie $n$-groupoid via the differential-geometric fat point a priori only yields a presheaf of graded manifolds. In this article, we prove that this presheaf is representable by the tangent complex of the Lie $n$-groupoid. As an immediate consequence, we obtain that the tangent complex carries the structure of a Lie $n$-algebroid.
\end{abstract}

%La différenciation d'un n-groupoïde via l'analgoue différentio-géométrique des points proches a priori ne donne qu'un pré-faisceau de variétés graduées. Dans cet article, nous démontrons que ce pré-faisceau est représentable par le complexe tangent du n-groupoïde. En conséquence immédiate, nous obtenons que le complexe tangent d'un tel n-groupoïde porte la structure d'une n-algèbre de Lie.



\maketitle

\tableofcontents

\section{Introduction}

Lie n-groupoids ($n=1, 2, \dots, \infty$) offer a clear and concrete way to represent geometric  $n$-stacks in differential geometry.  When $n=1$, they reduce to the well-studied Lie groupoids, which play a fundamental role  in differential geometry, topology, and non-commutative geometry. The infinitesimal counterpart of a Lie $n$-groupoid is a Lie $n$-algebroid, which can be understood as a kind of differential graded (d.g.) manifold concentrated in positive degrees, referred to  as  an NQ manifold. D.g. manifolds are extensively used in mathematical physics, particularly in BV theory, BRST theory, AKSZ theory, and topological field theory (TFT).

Therefore, how Lie $n$-groupoids and Lie $n$-algebroids {\bf correspond} to each other, becomes an important and interesting problem. Inspired by the relation between Lie groups and Lie algebras, which are Lie 1-groupoids over a point and Lie 1-algebroids over a point respectively,  we expect the following relation
\[ \xymatrix{
& \fbox{\parbox{.3\linewidth}{\center{Lie groups}}}
 \ar@{^{(}->}[d] \ar[rrrr]^{\text{differentiation}} &  & & &
  \fbox{\parbox{.3\linewidth}{\center{Lie algebras}}}  \ar@{^{(}->}[d]
  \ar[llll]^{\text{integration}}   \\
 & \fbox{\parbox{.3\linewidth}{\center{Lie $n$-groupoids}}} \ar[rrrr]^{\text{differentiation}} &  & & &
   \fbox{\parbox{.3\linewidth}{\center{Lie $n$-algebroids}}}
  \ar[llll]^{\text{integration}} }
  \]
When all the Lie brackets vanish, and the base manifold is simply a point, a Lie $n$-algebroid is just a chain complex of vector spaces of length $n$. Thus, the Dold-Kan correspondence may be viewed as the abelian case of the integration and differentiation correspondence. 
  \[
  \xymatrix{
 & \fbox{\parbox{.3\linewidth}{\center{Simplicial vector spaces\footnotemark }}} \ar[rrrr]^{\text{Dold-Kan}} &  & & &
   \fbox{\parbox{.3\linewidth}{\center{Chain complexes}}}
  \ar[llll]^{\text{correspondence}}
  }
\]\footnotetext{Notice that a simplicial vector space always satisfies Kan conditions \cite{Moore56}. 
}

In fact, even prior to its significant appearance in homological algebra, the problem of integration and differentiation is deeply rooted in differential geometry. This is evident from Lie's three theorems as well as from Van Est's theory. In recent times, there have been several significant milestones in this area. A few of these include: Integration of Lie algebroids \cite{cf}, classifying the first complete integrability obstruction for Lie algebroids; integration of Poisson manifolds examined from the perspective of sigma models \cite{cafe}, creating a link to mathematical physics; integration achieved via stacky Lie groupoids that surpasses integrability obstruction from a higher point of view \cite{tz}, opening a link to higher structures; and integrability obstruction and surpassing construction via higher structures for infinite-dimensional Lie algebras \cite{neeb:cent-ext-gp, WoZh16}.  As shown in these cases, differentiation  is straightforward when $n=1$, but integration poses a more complex and compelling challenge. This is because, unlike finite-dimensional Lie algebras, many Lie-algebra-like structures cannot be integrated. However, in the higher case when $n\ge 2$,  providing a universal differentiation is itself an unsolved problem. This is exactly the topic of our project. 

Indeed, the higher aspects of integration and differentiation have become a very active area of research in recent years. This line of work began with Getzler’s integration of nilpotent $L_\infty$-algebras \cite{getzler} using a construction parallel to the space realisation in rational homotopy theory. Continuing this line, Henriques \cite{henriques} subsequently studied the integration of general $L_\infty$-algebras via an infinite dimensional construction, while \v{S}evera and Siran \cite{sev:int} focused on the local integration of $L_\infty$-algebroids. Notably, going further from the integration of $L_\infty$-algebras as functor of iCFO (incomplete category of fibrant objects) developped by Rogers--Zhu \cite{Rogers-Zhu:2016}, there is a new recent progress by Rogers and Wolfson \cite{rogers-wolfson:24}, which provides a finite dimensional integration for finite dimensional $L_\infty$-algebras via Postnikov-tower and $k$-invariant technique.

From the differentiation perspective, Nuiten \cite{NuitenPhd}, building on Lurie’s theory of formal moduli problems \cite{LurieDAGX}, formulates differentiation in derived geometry at the level of $\infty$-categories. This is further developed in a recent preprint \cite[Theorem 3.2]{nuiten:26}.  A related approach in the same derived-geometric spirit was developed by Pridham, where ``integration/differentiation'' is implemented via a zig--zag of Quillen functors between model categories, using a variant of the Dold--Kan correspondence for (almost) cosimplicial algebras \cite[Section~3.2]{prid:shift}. In a recent article, Rogers \cite{rogers:2025} adopts a non-derived and finite-dimensional framework, and in return obtains a \emph{strict} 1-functorial differentiation theory for a Lie $\infty$-group $\huaG$. In particular, his construction is equipped with canonical identifications with the classical differentiation of a simplicial Lie group and provides a canonical isomorphism between the differentiation of $\huaG$ and its tangent complex \cite[cf.\ Lem.~3.14]{rogers:2025}, making the differentiation object tractable via $k$-invariants.  
%We also note two other very nice works of Cabrera--Del Hoyo \cite{Ca-DelHo:25} and Dorsch \cite{Dorsch:thesis} developing a Van Est theory for Lie $n$-groupoids, going beyond differentiation to relate global $n$-groupoid data to the corresponding infinitesimal invariants. 
We also note two other current series of works of Dorsch \cite{dorsch:24, Dorsch:thesis} and Cabrera–Del Hoyo \cite{Ca-DelHo:25}, which in particular include a differentiation theory for all simplicial manifolds, thanks to Dorsch's observation \cite{dorsch:24} that simplicial manifolds always satisfy local Kan conditions, as part of a broader Van Est program. Their results go beyond differentiation, providing a Van Est theory that compares global data with the corresponding infinitesimal invariants.
In the strict setting, Angulo and Cueca \cite{angulo-cueca:24} have recently given an explicit Van Est map for strict Lie $2$-groups. From a categorical viewpoint, Aintablian and Blohmann \cite{aint-bloh:25} develop a general differentiation formalism for differentiable groupoid objects and identify their infinitesimal counterparts in terms of associated (abstract) Lie algebroids.

\medskip

More broadly, there is a rich literature on integration and differentiation of Lie-theoretic objects endowed with additional geometric structure, especially in Poisson geometry and its generalizations: Poisson manifolds and symplectic groupoids \cite{Xu95, CrFe04}, Jacobi structures and contact groupoids \cite{cz}, (twisted) Dirac structures and 1-shifted symplectic groupoids \cite{BCWZ}, genus integration of  Contreras-Fernandes \cite{con-fer:18},  as well as Courant algebroids \cite{libland-severa1, MeTa11, MeTa18a, MeTa18b, shengzhu3} and a recent work of \'Alvarez--Gualtieri--Jiang \cite{alvarez-gualtieri-jiang} addressing the integration problem for generalized K\"ahler structures via symplectic double groupoids. As the relevant geometric structures become intrinsically ``higher''---for instance, Courant-type structures can be viewed as higher (shifted-symplectic) Lie-algebroid objects---the corresponding integration and differentiation problems move into the higher domain. In such settings, having differentiation and integration formulas in an explicit and Lie theoretical form is particularly important for concrete computations and applications.

In \cite{severa:diff},  using the pair groupoid of the ``odd line'', \v{S}evera  proposed a differentiation of 
Lie $n$--groupoids to presheaves on differential graded manifolds. This work also outlines an approach to representability of the differentiation, namely how to pass from presheaves on graded manifolds to genuine graded manifolds. This idea was further developed in \cite{li:thesis} and \cite{cech:2016}, effectively reducing the problem to a subtle and challenging combinatorial one. We note that \cite[Section~4.3]{cech:2016} contains a combinatorial gap\footnote{More precisely, around Eqs.~(4.17) and (4.18).}, as communicated to the authors, and that Lemma~8.34 in \cite{li:thesis} exhibits a similar issue.

In this article, we devote the entire Section \ref{sec:key-A} to solving the above combinatorial problem, and thus provide a first complete proof for the conjecture from \cite{severa:diff} that the tangent presheaf of an Lie $n$--groupoid is representable. In addition, unlike previous proof attempts, our proof does not go through local coordinates and rather relies on  connections and jet bundles. We develop a rigorous theory on graded manifolds involving  their limits, splittings, and compatible connections.  This leads to a very natural appearance of the tangent complex as the differentiation, and illuminates the differential geometric nature of the differentiation problem.

As mentioned previously, we aim at providing a universal differentiation to all Lie $n$-groupoids. Using \v{S}evera's pair groupoid of the ``odd line'' as a ``fat point'' $D_\bullet$ in the world of Lie $n$-groupoids, we give an explicit formula of the resulting tangent object. ``Fat point'' is a well-known mathematical slang, coming probably from Weil\footnote{This should have come from Andr\'e Weil's local algebra and his idea of ``points proches'' \cite{weilproche}. At least this is the earliest literature that the authors can trace back to this idea.}, 
 appears mostly in the context of algebraic geometry. For example, $P_\epsilon:=Spec(\R[x]/\langle x^2 \rangle)$ is the fat point in the world of $\R$-varieties: shooting with $P_\epsilon$ we obtain  tangent spaces of a variety (see \cite[Section V]{perrin:ag2008}). 
We prove in Theorem \ref{thm:ker-p} that shooting with the simplicial ``fat point'' $D_\bullet$ to a Lie $n$-groupoid $X_\bullet$,  one obtains exactly the  tangent complex of $X_\bullet$.
That is, 
\begin{thm}(Theorem [\ref{thm:ker-p}] Let $D_\bullet$ be the simplicial nerve of the pair groupoid of $D$ (cf. Examples \ref{ex:fatpoint} and \ref{ex:pair}), and $X_\bullet$ a Lie $n$-groupoid. Then  
\begin{equation} \label{eq:ker-p}
	\Hom(D_\bullet, X_\bullet) \cong \ker Tp^1_0|_{X_0}[1] \oplus \ker Tp^2_0|_{X_0}[2] \oplus \dots \oplus \ker Tp^n_0|_{X_n}[n], 
\end{equation}
where $p^k_0: X_k \to \Horn{n}{0}(X)$ is the horn projection (see Eq.~\eqref{eq:horn-proj}). 
\end{thm}

The tangent complex of a Lie $n$-groupoid plays the role of the tangent object in the world of Lie $n$-groupoids, thus also differentiable $n$-stacks. Its first explicit formula in quotient form appears in \cite{Lesdiablerets, MeTa18b}. Here, what we use in \eqref{eq:ker-p} is an isomorphic model given in \cite{cueca-zhu}. This formula is also observed by Du Li \cite{li:thesis} in a local form and in \cite{cech:2016}.  With this form, it is clear that the tangent complex of a Lie $n$-groupoid is a graded vector bundle of  degree $n$. As defined in Def. \ref{defi:T}, this procedure is functorial in $X_\bullet$. Thus, $D_\bullet$ serves well as a ``fat point'' in the world of Lie $n$-groupoids.
After proving the theorem, we apply our method to various examples, including the Lie 2-groupoids of Courant algebroids \cite{libland-severa1, MeTa11, MeTa18a,shengzhu3}, string Lie 2-groups, a form of higher groups called $n$-tower to explain higher topological orders in condensed matter physics \cite{Lan-Zhu-Wen:2019, Zhu-Lan-Wen:2019},  and simplicial Lie groups \cite{Jurco}.\\

Since the $D$-action on $\Hom(D_\bullet, X_\bullet)$ induces a natural homological vector field (\cite{severa:diff}), an immediate corollary of our result is 
\begin{cor}(Corollary \ref{cor:Q-structure})
	The tangent complex of a Lie $n$-groupoid carries the structure of a Lie $n$-algebroid. 
\end{cor}

As we mentioned, a nice and unexpected development, obtained a year later of our preprint by Dorsch \cite{dorsch:24}, is that our method applies to the differentiation of simplicial manifolds in full generality, thanks to the fact proved there that simplicial manifolds satisfy suitable local Kan conditions\footnote{This is also observed later in \cite{Ca-DelHo:25}.}.  Independently of this, but using overlapping tools and techniques, we plan to compute the higher Lie brackets explicitly. This will require a careful choice of the connections and splittings entering the isomorphism \eqref{eq:ker-p}, as well as a better understanding of the Lie $n$-algebroid structure on the iterated shifted tangent bundle. Conceptually, one may hope to reduce or even avoid such auxiliary choices by working instead with canonical multivector-bundle models \cite{madeleine-malte}; this idea appeared in an earlier version of the present article, but we now defer it to a separate future project. \\

{\bf Acknowledgement} We thank enormously Rui Fernandes, who proposed to us a method using $S_n$ symmetry based on the technique in \cite{rui} to reduce the complexity of the combinatorial problem dramatically. This method would have almost solved the problem, but we discovered that at a certain place it could not generalize to Lie $n$-groupoids when $n\ge 2$. 
We would like to thank very much Christian Blohmann and Lory Aintablian, who shared with us their insights on tangent functors in general categories.
We also give our warmest thanks to Miquel Cueca, Florian Dorsch,  Madeline Jotz, Branislav Jurco, Joost Nuiten, Dimitry Roytenberg, Pavol Severa and Tilmann Wurzbacher. The authors thank the anonymous referee for their careful reading and suggestions for improving the manuscript. We are supported by DFG ZH 274/3-1, the Procope Project GraNum of the DAAD and Campus France, RTG 2491. Finally, we thank very much the ESI Vienna program Higher Structures and Field Theory, Huazhong University of Science and Technology, and the IHP Paris Trimester  Higher Structures in Geometry and Mathematical Physics,  where the project was worked on. 

\section{Lie $n$-groupoids and graded manifolds}

\subsection{Simplicial manifolds and their horns}

A {\bf simplicial manifold}
$X_\bullet$ is a contravariant
functor from $\Delta$, the category of finite ordinals 
\begin{equation}\label{eq:ord}
   [0]=\{0\}, \qquad [1]=\{0, 1\},\quad \dotsc,\quad
[l]=\{0, 1,\dotsc, l\},\quad\dotsc, 
\end{equation}
with order-preserving maps, to the category of manifolds. More precisely, $X_\bullet$ consists of a tower of manifolds $X_l$, face maps $d^l_k: X_l \to X_{l-1}$ for $k=0, \dots, l$, and degeneracy maps $s^l_k: X_l \to X_{l+1}$. These maps satisfy the following simplicial identities
\begin{equation}\label{eq:face-degen}
    \begin{array}{lll}
        d^{l-1}_i d^{l}_j = & d^{l-1}_{j-1} d^l_i &\text{if}\; i<j,  \\
       s^{l}_i s^{l-1}_j =& s^{l}_{j+1} s^{l-1}_i & \text{if}\; i\leq j,
    \end{array}\qquad  d^l_i s^{l-1}_j =\left\{\begin{array}{ll}
    s^{l-2}_{j-1} d^{l-1}_i  & \text{if}\; i<j, \\
    \id  & \text{if}\; i=j, j+1,\\
    s^{l-2}_j d^{l-1}_{i-1} & \text{if}\; i> j+1.
 \end{array}\right.
\end{equation}
We may drop the upper indices and just write $d_i$ and $s_i$ for simplicity when the context is clear later in our article. 

Similarly, we may define other simplicial objects in other categories, such as simplicial sets, simplcial vector spaces, which we will meet in this article. For any category $\mathcal C$, we denote by $\simp{\mathcal C}$ the category of its simplicial objects.

\begin{example}[Pair groupoids in $\mathcal C$]\label{ex:pair}
In any category $\mathcal C$, which admits finite products, for any object $X\in \mathcal C$, there exists a pair groupoid $X\times X\rightrightarrows X$, whose nerve we denote by $X_\bullet$. It can be seen as the functor $[n]\mapsto X^{[n]}\cong X\times .... \times X$ (n+1 copies).
\end{example}

\begin{example}[Discrete groupoids in $\mathcal C$]
In any category $\mathcal C$, for any object $X\in \mathcal C$, the functor $[n]\mapsto X$ defines a simplicial manifold (with all faces and degeneracies being the identity). It is distinct from the nerve of the pair groupoid and, somewhat abusively, we will simply denote it $X\in \simp{\mathcal C}$.
\end{example}

The following simplicial sets, which may be viewed as simplicial manifolds with discrete topology (e.g., in Def. \ref{def:lie-n-gpd}), play an important role for us. They are the simplicial {\bf $l$-simplex} $\Delta[l]$ and
the {\bf horn} $\Lambda[l,j]$,
\begin{equation}\label{eq:simplex-horn}
\begin{split}
(\Delta[l])_k & = \{ f: [k] \to [l] \mid f(i)\leq
f(j),
\forall i \leq j\}, \\
(\Lambda[l,j])_k & = \{ f\in (\Delta[l])_k\mid \{0,\dots,j-1,j+1,\dots,l\}
\nsubseteq \{ f(0),\dots, f(k)\} \}.
\end{split}
\end{equation}
In fact the horn $\Lambda[l,j]$ is a simplicial set obtained from the
simplicial $l$-simplex $\Delta[l]$ by taking away its unique
non-degenerate $l$-simplex as well as the $j$-th of its $l+1$
non-degenerate $(l-1)$-simplices, as in the following picture (in
this paper all the arrows are oriented from bigger numbers to
smaller numbers): \vspace{.6cm}

\centerline{\epsfig{file=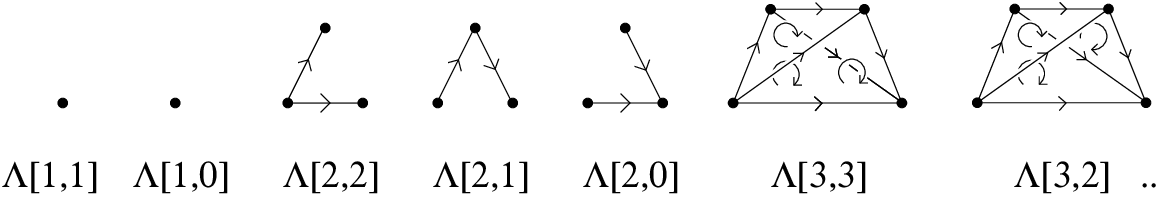,height=1.9cm}}

\begin{definition} \label{def:lie-n-gpd}
 A \textbf{Lie $n$-groupoid} \cite{getzler, henriques, z:tgpd-2} is a simplicial manifold $X_\bullet$ where the natural projections
\begin{equation}\label{eq:horn-proj}
 p^l_{j}:X_l=\Hom(\Delta[l],X_\bullet)\to \Hom(\Lambda[l,j], X_\bullet)=: \Horn{l}{j}X
\end{equation}
 are surjective submersions\footnote{Submersions for Banach manifolds is also introduced in \cite{lang}. It is sometimes called a split submersion in other literature, namely we need local splitting in charts.} for all $1\le l\ge j\ge 0$ and diffeomorphisms for all $0\le j\le l>n$. It is further a \textbf{Lie $n$-group} if $X_0=pt$.  We call the space  $\Horn{l}{j}X$ the $(l,j)$-\textbf{horn} and the projections $p^l_{j}$ \textbf{horn projections}.
 
If we replace surjective submersions by submersions, and isomorphisms by injective \'etale maps, in the requirement for the natural projections $p^l_{j}$, we say that $X_\bullet$ is a \textbf{local Lie $n$-groupoid}, see \cite[Defi.1.1]{zhu:kan}. Similarly, it is further a \textbf{local Lie $n$-group} if $X_0=pt$. 
\end{definition}

\begin{remark}\label{remark:1-2}
A simplicial map between two simplicial sets \(S\) and~\(T\) is a
family of maps \(S_n\to T_n\) that intertwine the face and
degeneracy maps on \(S\) and~\(T\).  If we fix~\(S\), we may view
the set of all simplicial maps \(S\to T\) as the limit of a certain
diagram; the diagram is indexed by the category~\(\mathcal I\) with
objects \(\bigsqcup S_n\) and arrows \((s, f)\colon s\mapsto
f^*(s)\) for all \(s\in S_n\), \(f\colon [m]\to [n]\); the diagram
maps \(s\in S_n\) to the corresponding element in~\(T_n\) and the
arrow \((s,f)\) to \(f^*\colon T_n\to T_m\).  If we replace~\(T\) by
a simplicial object in a category~\(\Cat\), then this defines a
diagram in~\(\Cat\).  We let \(\Hom(S,T)\) be its limit in~\(\Cat\),
if it exists (in general, it may be defined only as a presheaf
over~\(\Cat\)). Or in the case of a {\bf Lie $n$-groupoid}, we can roughly view simplicial sets $\Delta[m]$ and $\Lambda[m,j]$ as
simplicial manifolds with their discrete topology so that $\Hom(S,X)$
denotes the set of homomorphisms of simplicial manifolds, with its
natural topology resp. smooth structure. 
Thus, $\Hom(\Delta[m],X)$ is just another name for $X_m$.
However, it is not obvious that $\Hom(\Lambda[m,j],X)$ is again a manifold, and it is a result of \cite[Corollary 2.5]{henriques}.
\end{remark}

\begin{example}[Lie $1$-groupoids]\label{ex:1group}
It is well known that any Lie $1$-groupoid $K_\bullet$ is always the simplicial nerve $NK_\bullet$ of a usual Lie groupoid $K\rightrightarrows M$. Here we recall that $NK_0=M$ and for $i\geq 0$ define $NK_i=\{(k_1,\cdots, k_i)\in K^{\times i}\ |\ s(k_j)=t(k_{j+1})\}$. The faces and degeneracies are given by  \begin{equation*}
    \begin{array}{rlrl}
        d_0(k)=&s(k), & d_1(k)=&t(k),\\
        d_0(k_1, \dots, k_i)=&(k_2,\dots ,k_i),& d_j(k_1, \dots, k_i)=&(k_1,\dots,  k_jk_{j+1}, \dots, k_i),\\
        d_i(k_1, \dots, k_i)=&(k_1, \dots, k_{i-1}),
         &  s_j(k_1, \dots, k_i) =& (k_1, \dots, k_j, 1_{s(k_j)}, k_{j+1}, \dots, k_i).
    \end{array}
\end{equation*} 
\end{example}

\begin{example} [Lie 1-group]\label{NG}
A particular case of the previous example is when we consider a Lie group $G$. Then its simplicial nerve $NG_\bullet$ is given by $NG_i=G^{\times i}$ and the faces and generacies are just
\begin{equation*}
    \begin{array}{ll}
        d_0(g_1, \dots, g_i)=(g_2,\dots ,g_i),& d_j(g_1, \dots, g_i)=(g_1,\dots,  g_jg_{j+1}, \dots, g_i),\\
        d_i(g_1, \dots, g_i)=(g_1, \dots, g_{i-1}),
         &  s_j(g_1, \dots, g_i) =(g_1, \dots, g_j, e, g_{j+1}, \dots, g_i).
    \end{array}
\end{equation*} 
\end{example}

\subsection{Graded manifolds and their tangent spaces}
We are interested in two sorts of graded manifolds, the first of which are N-manifolds. Our objects of interest --- Lie $n$-algebroids are N-manifolds equipped with a degree 1 differential $Q$, also known as a homological vector field. However, we will need our version of the fat point and the test object to be $\Z$-graded manifolds. 

{Various flavours of $\Z$-graded manifolds exist throughout the literature, cf. \cite{Voronov02} or \cite{Vysok2022, kotov-salnikov:21}}. Here we give a short recall with a slightly modified definition of $\Z$-graded manifolds according to our needs in this article, while the one for N-manifolds is equivalent to the standard references  (see \cite[Section1.1]{Zambon-Zhu:quotient} and references therein). 
\begin{definition}\label{def:z-graded}
A $\Z$-{\bf graded manifold}  $\huaM$ is 
a $\Z$-graded graded-commutative 
algebra $C(\huaM)$ (called the ``function algebra'') 
such that there exist a $\Z$-graded vector bundle $V_\bullet$ and an isomorphism 
\begin{equation} \label{eq:iso-gmfd}
    C(\huaM)\cong \Gamma(S^{\ge 0} V^*_\bullet). 
\end{equation}
Furthermore, when $V_\bullet$ is concentrated on strictly negative degrees, we call $\huaM$ an $N$-manifold. 
Throughout the article, unless otherwise mentioned,  we always assume that a graded vector bundle $V_\bullet$ has 
finite rank in each component $V_i$.
\end{definition}
\begin{remark}\label{rmk:z-graded}
In this definition and throughout this article, $S^{\ge 0}$ denotes the (augmented) symmetric algebra, that is, it includes the arity 0 part $\Gamma(S^{0} V^*_\bullet) := C^\infty (M)$, which has degree 0. Otherwise, the degree of an element in $\Gamma(S^{\ge 0} V^*_\bullet)$ is the sum of the degrees of each term in the symmetric product.

For convenience, we define graded manifolds using just the global function algebra and not using a sheaf of algebras. This way, the Batchelor-Gawedzki Theorem, which has certain subtleties in the $\mathbb Z$-graded case (cf. \cite{kotov-salnikov:21}), is part of the definition for us.
\end{remark} 

\begin{definition}\label{defi:z-graded-mor}
    A morphism of $\Z$-graded manifolds $\huaM \to \huaN$ is 
    a $\Z$-graded algebra 
    morphism $\Phi: C(\huaN)\to C(\huaM)$.
\end{definition}

\begin{remark}{We define our maps purely as algebra homomorphisms, without an underlying base map, because for our later differentiation construction we will need that morphisms can map smooth coordinates (i.e., those coming from $C^{\infty}(M)$) to formal ones (i.e., degree 0 elements in $\Gamma(S^{\geq 1}V^*)$)}.
\end{remark}

We denote the category of $\Z$-graded manifolds by $\Z\Mfd$ and the category of N-manifolds by $\NMfd$, the subcategory of graded manifolds whose vector bundles have non-trivial components only in strictly positive degree. {Note that the function algebras of N-Manifolds are not strictly positively graded: In degree 0, they have the algebra of smooth functions of the underlying manifold.}
Since morphisms preserve degrees,  $\NMfd$ is a fully faithful subcategory of $\Z\Mfd$.

\begin{remark} \label{rm:strict}
\quad
\begin{enumerate}
    \item 
Clearly, a $\Z$-graded vector bundle $V_\bullet$ gives us a $\Z$-graded manifold $\mathfrak{S}(V_\bullet)$, and 
\begin{equation} \label{eq:S}
    \mathfrak{S}: \Z\Vectbd\to \Z\Mfd
\end{equation}
is a functor from the category of $\Z$-graded vector bundles to that of $\Z$-graded manifolds. 
Given a $\Z$-graded manifold $\huaM$, a choice of a vector bundle $V_\bullet$ together with a choice $i_{V_\bullet}$ of the isomorphism \eqref{eq:iso-gmfd}, 
that is, $(V_\bullet, i_{V_\bullet})$,  is called a {\bf splitting} of $\huaM$. However, later in the article, when there is no confusion of choice of isomorphisms in the context, we also refer to $V_\bullet$ alone as a splitting of $\huaM$.
\item We call a morphism $\Phi: \huaM \to \huaN$ between $\Z$-graded manifolds {\bf strict} with respect to some splitting of $\huaM$ and $\huaN$ (or simply strict for simplicity when the context is clear) if it comes from a morphism of the corresponding $\Z$-graded vector bundles. More precisely, there are a splitting $V_\bullet$ of $\huaM$ and $W_\bullet$ of $\huaN$ and a $\Z$-graded vector bundle morphism $f_\bullet: V_\bullet \to W_\bullet$ such that $S^{\ge 0}(f_\bullet)$ is the composed map 
\begin{equation*}
    \Gamma(S^{\ge 0}W_\bullet^*) \xrightarrow{\cong} C(\huaN) \xrightarrow{\Phi} C(\huaM) \xrightarrow{\cong} \Gamma(S^{\ge 0}V_\bullet^*). 
\end{equation*}

\item Moreover, $\mathfrak{S}$ clearly restricts to $\Z^{<0}$-graded vector bundles, gives rise to a functor 
\begin{equation*}
    \mathfrak{S}: \Z^{<0}\Vectbd\to \NMfd
\end{equation*}
from the category of $\Z^{<0}$-graded vector bundles to that of N-manifolds. We use the same terminology of  ``splitting'' and ``strict'' as above. 

\item The concept of a strict morphism involves splittings of the graded manifolds. Thus, the composition of strict maps is {\bf not} necessarily strict. We need the splittings of the  graded manifold in the middle to be the same, because an automorphism of a graded manifold does {\bf not} necessarily come from a vector bundle morphism between two different splittings.  

\item There exist morphisms of graded manifolds that are not strict, no matter which splittings we choose. Here is an example: We consider the graded manifold $\mathcal M=\mathbb R^2[1]$, whose function algebra is generated by two elements $\epsilon_1,\epsilon_2$ in degree 1 and the graded manifold $\mathcal N=\mathbb R[2]$, whose function algebra is generated by one element $\delta$ in degree 2. The map $\delta\mapsto \epsilon_1\epsilon_2$ defines a non-trivial morphism $\mathcal M\to\mathcal N$; however, any strict morphism from $\mathcal M$ to $\mathcal N$ would be trivial for degree reasons.

\item Strict embeddings are clearly monomorphisms in either $\Z\Mfd$ or $\NMfd$. Recall that a {\bf monomorphism} in a category is a left-cancellative morphism. That is,  $ X \xrightarrow{f} Y$ is a monomorphism if for all objects $Z$ and all morphisms $g_1, g_2:  Z \to X$,  $f\circ g_1=f\circ g_2$ implies $g_1=g_2$.
\end{enumerate}
\end{remark}

Given a graded manifold $\mathcal M$, the manifold $M$ and the graded vector bundle $V\to M$ such that $ C(\mathcal M)\cong \Gamma(S^{\geq 0}V^*)$ are unique up to isomorphism. We call $M$ the {\bf body} of $\mathcal M$, and we say that $\huaM$ is over $M$. When one of the graded manifolds $\mathcal M,\mathcal N$ is over a point, we define  their {\bf product} by the usual tensor product $C(\mathcal M\times \mathcal N)= C(\mathcal M)\otimes  C(\mathcal N)$. 

\begin{remark}
Notice the isomorphism \eqref{eq:iso-gmfd}  is highly non-canonical. Thus, the inclusion\footnote{In fact, neither the inclusion from $C^\infty(M)$ to $C(\mathcal M)$ nor the projection in the other direction in general is canonical.} from $C^\infty(M)$ to $C(\mathcal M)$ is not canonical, unless $M$ is a point.  This makes defining products of graded manifolds $\mathcal M,\mathcal N$ intrinsically in general a little subtle\footnote{One can of course pick splittings and construct the graded manifold corresponding to the product of the graded vector bundles. But then one will need to pay attention to choices of splittings.}. Since we only need products of graded manifolds in the situation when one of them is over a point, we give the above definition and defer the general definition to future work. 
\end{remark}

\begin{example} \label{ep:T1M} Let $M$ be a manifold. We denote by $T[1]M$ the $\Z$-graded manifold whose function algebra is $\Omega^\bullet(M)$. It has a canonical splitting, given by vector bundle $V_\bullet=V_{-1}=TM$ in degree -1.\\

More generally, for a $\Z$-graded or $\Z^{<0}$-graded vector bundle  $V_\bullet$,  we denote $V[k]_\bullet$ (or $V_\bullet[k]$ by abuse of notation sometimes) its $k$-shift (to the left), that is, $V[k]_i = V_{k+i}$.
In particular, if $V$ is a vector bundle, $V[k]$ is a $\Z$-graded vector bundle which concentrates in degree $-k$. 

Thus, with this notation, $T[k]M$ then has a double meaning: it is the $\Z$-graded manifold corresponding to $\Omega^\bullet(M)$ with a suitable shift of degree,  and it is also a $\Z$-graded vector bundle with $TM$ on degree $-k$. That is,  $T[k]M=\mathfrak{S}(T[k]M)$. We hope that this little confusion is harmless for readers.
\end{example}

\begin{example}\label{ex:fatpoint}
We will denote the split graded manifold $\mathbb R[-1]$ by $D$, where $\mathbb R$ is the 1-dimensional trivial vector bundle over a point. We write its function algebra as $C(D)=\mathbb R[\epsilon]$ with $deg(\epsilon)=-1$ and hence $\epsilon^2=0$. Careful: $\Z$-graded manifold $\mathbb R[-1]$ is very different from $\mathbb R[1]$, even though both of them look like an odd line in super geometry, where objects are $\Z/2\Z$-graded manifolds.
\end{example}

In algebraic geometry, the concept of ``fat point'' is used to detect tangent objects {\cite{perrin:ag2008, weilproche}}. 
We will see that $D$ plays the role of a graded and 1-shifted version of the ``fat point'' for differential geometry. In order to make sense of it, we require the notion of presheaves.

Let $\mathcal C$ be a category. We denote by $\pSh{\mathcal C}$ the category of {\bf presheaves} (Set-valued contravariant functors) on $\mathcal C$. There is a functor $X\mapsto \hom(\cdot , X):\mathcal C\to \pSh{\mathcal C}$ and we call a presheaf {\bf representable}, if it is isomorphic to something in the image of this functor. Furthermore, for $A,B\in \mathcal C$, we define  a presheaf on $\huaC$, when the product of $A$ with any object in $\mathcal C$ exists, 
\begin{equation}\label{eq:Hom}
    \Hom(A, B): T\mapsto \hom(A\times T, B),
\end{equation}
where $T\in \mathcal{C}$ is an arbitrary test object and the uncapitalized $\hom$ denotes the set of morphisms in $\huaC$. Using this notion, we can define the tangent object of a graded manifold. We have the following theorem:

\begin{thm}\label{thm:splitgrad} Let $\mathcal M$ be $\Z$-graded manifold with a splitting $V_\bullet\to M$. Then the presheaf $\Hom(D, \mathcal M)$ is representable and a connection on $V_\bullet^*$ induces a splitting of $\Hom(D, \mathcal M)$ by $V_{\bullet}\oplus T[1]M\oplus V[1]_{\bullet}$. Thus, we denote $\Hom(D, \mathcal M)$ by $T[1]\huaM$, the 1-shifted tangent bundle of $\huaM$. 
\end{thm}

{This theorem seems folklorically known to the graded geometry community, but we failed to find a precise reference, so we provide a proof here. Moreover, we will need to keep precise track of the spot where we use the connection in order to investigate the compatibility between splittings and morphisms later.}

\begin{proof}
Consider a morphism in $Hom(D,\mathcal M)(\huaT)=\hom(D\times \mathcal T,\mathcal M)$, where $\mathcal T$ is a $\Z$-graded manifold with body $T$. Since the body of $D$ is simply a point, such a morphism is given by an algebra morphism $\Psi:C(\mathcal M)\to \mathbb R[\epsilon]\otimes C(\mathcal T)$.  For $\alpha\in C(\mathcal M)$ we can decompose $\Psi$ as
$$
\Psi(\alpha)=\Psi_0(\alpha) + \epsilon\cdot \Psi_{1}(\alpha),
$$
where $\Psi_0$ and $\Psi_{1}$ map to $C(\huaT)$. The algebra $C(\mathcal M)$ is generated by elements $f\in C^\infty(M)$ and $\sigma\in \Gamma(V_k^*), k\in \mathbb Z$. We can translate the fact that $\Psi$ is an algebra homomorphism to the following equations:
\begin{align}
 &\Psi_0(f g)=\Psi_0(f)\Psi_0(g)  \label{eq:T-M} \\
  &\Psi_0(f \sigma)=\Psi_0(f)\Psi_0(\sigma)\\
  &\Psi_{1}(fg)=\Psi_0(f)\Psi_{1}(g)+\Psi_{1}(f)\Psi_{0}(g) \label{eq:Psi-0-1-fg}\\
  &\Psi_{1}(f\sigma)=\Psi_0(f)\Psi_{1}(\sigma)+\Psi_{1}(f)\Psi_{0}(\sigma),  \label{eq:Psi-1-f-sigma}
\end{align}
for $f, g \in C^\infty(M)$ and $\sigma \in \Gamma(V^*_k)$, $k\in \Z$.  The first two equations imply that $\Psi_0$ gives rise to a morphism of graded manifold from $\huaT$ to $\huaM$ 
The last two equations show that $ \Psi_{1}$ satisfies a Leibniz rule over $\Psi_0$.\\

In the following, we will need some knowledge on jet bundles and differential operators and we give explicit references in \cite{Nestruev03}.  The fact that $\Phi$ is an algebra homomorphism implies that $\Psi_{1}$ satisfies  
\begin{equation}\label{eq:diff-1}
    [f, [g, \Psi_1]]=0, \quad \forall f, g \in C^\infty(M), 
\end{equation}
where $[f, P]$ denotes the operator $[f,P](s)=P(fs)-fP(s)$ for all $s\in C(\huaM)$. Eq.~\eqref{eq:diff-1} implies that $\Psi_{1}$ is a differential operator of order 1 in $\Diff_1(C^\infty(M)\oplus \Gamma(V^*_\bullet), C(\huaT))$ between the two $C^\infty(M)$-modules.  Here $C^\infty(M)$ acts on $C(\mathcal T)$ via $\Psi_0$.  We want to apply Theorem 14.25 from \cite{Nestruev03}, stating that differential operators of order $l$ from a projective module into a geometrical one correspond to module homomorphisms from the corresponding $l$-jet bundle. Clearly, the module $ \Gamma(V^*_k)$ (resp. $C^{\infty}(M)$) is projective, so we only need to verify that $C(\mathcal T)$ is geometrical as a $C^{\infty}(M)$-module. For an $\R$-algebra $A$, an $A$-module is called geometrical (section 12.11 in \cite{Nestruev03}) if the space of invisible elements is trivial:
$$
Inv(P):=\bigcap_{h\in \hom_{\alg}(A,\mathbb R)} \ker(h)\cdot P \overset{!}{=}0. 
$$

Now we show that $C(\mathcal T)$ is geometric as a $C^\infty(M)$-module. By  the naturality of geometrization (15.29 in \cite{Nestruev03}), we only need to show that that $C(\mathcal T)$ is a geometric $C_0(\mathcal T)$-module, since $C^\infty(M)$ acts on $C(\mathcal T)$ via $\Psi_0$. Since a projective module is always geometric  \cite[Section12]{Nestruev03}, using a splitting of $\mathcal T$, we can see that $C(\mathcal T)$ is a geometric $C^\infty(T)$-module. Moreover, we get morphisms $C^\infty(T)\overset{i}\to C_0(\mathcal T)\overset{p}\to C^\infty(T)$ composing to the identity. This means that the map $i^*:\hom_{\alg}(C_0(\mathcal T),\mathbb R)\to \hom_{\alg}(C^\infty(T),\mathbb R)$ is surjective. This implies (by going through the proof of 15.29 in \cite{Nestruev03} backwards) that $C(\mathcal T)$ is a geometric $C_0(\huaT)$-module.

Therefore Theorem 14.25 from \cite{Nestruev03} is applicable and $\Psi_{1}$ is uniquely determined by the $C^{\infty}(M)$-module map  $J^1\Phi_1: \Gamma(T^*M \oplus J^1V^*_\bullet) \to C(\huaT)$, with components, 
$$
J^1_{T^*}\Psi_{1}:\Gamma(T^*M)\to C_{1}(\mathcal T),~~ J^1_k\Psi_{1} :\Gamma(J^1V_k^*)\to C_{k+1}(\mathcal T),  \quad \forall k \in \Z. 
$$

Now we would like to see how much and how redundantly the information of $(\Psi_0, \Psi_{1})$ can be encoded using $(\Psi_0, J^1\Psi_{1})$. Notice that the space $J^1V_k^*$ contains $T^*M\otimes V_k^*$ as a subspace (see \eqref{eq:jet-split}), on which the values of $J^1_k\Psi_{1}$ are already determined by $\Psi_0$ and $J^1_{T^*} \Psi_{1}$. Hence, to get rid of the redundancy, we have to pick a complement of $T^*M\otimes V_k^*$ in $J^1V_k^*$, i.e., a splitting $s_k:V_k^*\to  J^1V_k^*$ of the following sequence of vector bundles over $M$:
\begin{equation}\label{eq:jet-split}
0\to T^*M\otimes V_k^*\to J^1V_k^* \to V_k^* \to 0, 
\end{equation}
or, in other words, a $TM$ connection on $V_k^*$ for each $k\in\Z$. We now fix a series of such connections $s_\bullet$. 

By property of 1-jet, $\Phi_1 = J^1\Phi_1 \circ j^1: \Gamma(V^*_k) \to C(\huaT)$. 
Let $\sigma \in \Gamma(V_k^*)$ and $f\in C^\infty(M)$. By definition $j^1(f\sigma)=df\otimes \sigma +fj^1\sigma$. Since $J_k^1\Psi_{1}$ is $C^{\infty}(M)$-linear with $C^\infty(M)$ acting on $C(\huaT)$ via $\Psi_0$,  \eqref{eq:Psi-1-f-sigma} implies that 
 $$J_k^1 \Psi_{1}(df\otimes \sigma)=\Psi_{1}(f)\Psi_{0}(\sigma).$$
So it suffices to know the values of $J^1_k\Psi_{1}$ on a complement of $T^*M\otimes V_k^*$, i.e., on the image of $s_\bullet$. This means that the information of $\Psi$ is uniquely encoded by:
\[
\Psi_0: \Gamma(V_\bullet)^* \to C(\huaT), \quad  J^1_{T^*}\Psi_{1}: \Gamma(T^*M)\to C_1(\huaT), \quad J^1_k \Psi_{1} \circ s_k: \Gamma(V^*_k) \to C_{k+1}(\huaT).
\]
Thus, $\Hom(D, \huaM)$ is representable and a connection on $V^*_\bullet$ gives a splitting of it by $V_\bullet \oplus T[1]M \oplus V[1]_{\bullet}$.   
\end{proof}

For ordinary manifolds, the notation introduced in this theorem is consistent with the existing one.

\begin{example}
    Let $\mathcal M=M$ be an ordinary manifold (with $C(\mathcal M)=C^\infty(M)$). Then $Hom(D,\mathcal M)$ is representable by $T[1]M$. This statement can be seen as a special case of the previous theorem. 
\end{example}

Inspired by the above theorem, for a $\Z$-graded vector bundle $V_\bullet$ over $M$, we define its 1-shifted tangent to be 
\[
T[1]V_\bullet := V_\bullet \oplus T[1]M \oplus V[1]_\bullet. 
\] Thus, Theorem \ref{thm:splitgrad} tells us that a connection on $V^*_\bullet$ gives us an isomorphism 
\begin{equation}\label{eq:T1V}
    T[1]\mathfrak{S}(V_\bullet) \cong \mathfrak{S} (T[1] V_\bullet).
\end{equation}

\subsection{Iterated tangent bundles for manifolds} \label{sec:iterated-tan}
The object of study for us will be the iterated tangent bundles of an ordinary manifold:
$$
Hom(D^k,M)=T[1]Hom(D^{k-1},M)=T[1]T[1]Hom(D^{k-2},M)=...=: T[1]^kM
$$
Now the simplicial and graded worlds start merging: Since the nerve of the pair groupoid
\begin{equation} \label{eq:pair-D}
    D_\bullet=\{D_{k}=D^{[k]}\}_{k\in \mathbb N}
\end{equation}
carries a simplicial structure, the iterated tangent bundles will carry a cosimplicial one, which---in much larger generality - is the object of study of \cite{loriChristian}. \\

\begin{cor}(of Theorem \ref{thm:splitgrad})
Let $M$ be a manifold with a fixed connection $\nabla$ on $T^*M$. The presheaf $T[1]^kM:=Hom(D^k,M)$ is representable, 
\begin{equation} \label{eq:T1-k-M}
    T[1]^k M:=\Hom(D^{\times k}, M) \cong \mathfrak{S}( T[1]M^{\binom{k}{1}} \oplus T[2]M^{\binom{k}{2}} \oplus \dots \oplus T[k]M^{\binom{k}{k}}), 
\end{equation} 
where $T[i]M^{\binom{k}{i}}:= T[i]M \oplus \dots \oplus T[i]M$ denotes ${\binom{k}{i}}$ copies of direct sum of $T[i]M$ over $M$.  
\end{cor}
\begin{proof}
This follows from  Example \ref{ep:T1M} and a repeated application of Theorem \ref{thm:splitgrad} with the observation made above Eq.~\eqref{eq:T1V}.  The equation ${\binom{k}{i}} + {\binom{k}{i-1}} = {\binom{k+1}{i}}$ will be repeatedly applied in the induction process.  
\end{proof}
In the sequel, we will have to do some combinatorics on \eqref{eq:T1-k-M}. For simplicity, let us denote the graded vector bundle on the right-hand side by $T^k_\bullet$, that is, 
\begin{equation}\label{eq:T-k}
    T^k_{-i}= \bigoplus_{\{I\subset [k-1]: |I|=i\}} TM_I,
\end{equation}
where $I$ as an $i$-element subsets of $[k-1]=\{0,...,k-1\}$ can be also understood as a strictly increasing $i$-multi-index with values in $[k-1]$, and $TM_I$ is a copy of $TM$ indexed by $I$. 
\begin{notation}\label{not:index}
We introduce some notation. Let $I\subset [k-1]$ be a multi-index. 
\begin{itemize}
    \item For $v$ in $TM$, we write $v^I$ for an element in $T^k_\bullet$, which is $v$ in the $I$-th component $TM_I$ and 0 in all other components.
    \item For  $i\in \Z$, we write $\overset{i\to}{I}$ for the multi-index with all indices in $I$ after $i$ pushed to the right by one. For example $I=(2, 3, 5, 7)$, then $\overset{3\to}{I}=(2, 3, 6, 8)$,  $\overset{4\to}{I}=(2, 3, 6, 8)$, $\overset{-1\to}{I}=(3, 4, 6, 8)$, and $\overset{10\to}{I}=(2, 3, 5, 7)$
    \item Similarly,  $\overset{i\leftarrow }{I}$ denotes the index where everything in $I$ after $i$ is pulled by one to the left, e.g., $J=(2, 3, 6, 8)$, then $\overset{2\leftarrow }{J}=(2, 2, 5, 7)$, $\overset{3\leftarrow }{J}=(2, 3, 5, 7)$ and $\overset{4\leftarrow }{J}=(2, 3, 5, 7)$. We notice that while $\overset{i\to}{I}$ is still a strictly increasing $|I|$-multi-index, $\overset{i\leftarrow }{I}$ might not be such a multi-index anymore because a number can appear twice, as the example shows.  In such cases, when $\overset{i\leftarrow }{I}$ is no longer a strictly increasing $|I|$-multi-index, we take $v^{\overset{i\leftarrow }{I}}=0$ as there is no such component in the sum of \eqref{eq:T-k}.
\end{itemize}
\end{notation}

This immediately allows us to translate the simplicial identities on $D_\bullet$ to co-simplicial identities on $T[1]^\bullet M$ with splitting \eqref{eq:T-k}.

\begin{lemma} 
The face map $d_j:D^{[k]}\to D^{[k-1]}$, for $j\in [k]$ corresponds to the coface map $d_j^*:T^{k-1}_\bullet\to T^{k}_\bullet$, and the degeneracy maps $s_i:D^{[k-1]}\to D^{[k]}$, for $i\in [k-1]$ corresponds to the codegeneracy map $s_i^*:T^{k}_\bullet\to T^{k-1}_\bullet$. They re-arrange the indices as follows:
\begin{equation*}
    d^*_j: \xymatrix@=0.3cm{0, \ar[d] & 1, \ar[d] & \dots, & j-1, \ar[d] & j, \ar[dr] & \dots, & k-1 \ar[dr]\\
    0,& 1, &\dots & j-1, &j, & j+1, & \dots, & k}, \quad s_i^*: \xymatrix@=0.3cm{0, \ar[d] & 1, \ar[d] & \dots, & i-1, \ar[d] & i, \ar[dl] & \dots, & k \ar[dl]\\
    0,& 1, &\dots  &i, & \dots, & k-1}. 
\end{equation*}
More precisely, for $j\in [k]$ and $i\in [k-1]$
\begin{equation}
    d^*_j(v^I)=v^{\overset{j-1\to}{I}}, \quad s_i^* (v^I) = v^{\overset{i\leftarrow}{I}},
\end{equation}
where by convention $v^{\overset{i\leftarrow}{I}}=0$ when $i,i+1\in I$.
\end{lemma}

\subsection{Limits for N-manifolds}\label{sec:limit}

It is relatively complicated to describe limits explicitly in the category of graded manifolds. Notice that there are morphisms (see Remark \ref{rm:strict}) between graded manifolds which do not come from morphisms between graded vector bundles, but some sort of non-strict $L_\infty$-style maps. It is not clear how to describe a fiber product, for example, if one leg is such a non-strict map. However, pullbacks of graded vector bundles can be easily described if one leg is a submersion on the base. In fact, the underlying sets for limits of graded vector bundles can be easily explicitly described, as every graded vector bundle has an underlying set. Luckily, in the case of N-manifolds, which is the case we need in this article, 
the functor $\mathfrak S : \Z^{<0}\Vectbd\to \NMfd$ preserves limits. Therefore, when we have strict morphisms, it is easy to describe a limit of N-manifolds via the explicit formula in the world of $\Z^{<0}$-graded vector bundles.

We now set off to show this fact. For this, we introduce the notion of vector bundle comorphisms (\cite{higmac93}).

\begin{definition}
    Let $E\to M$ and $F\to N$ be vector bundles. A vector bundle comorphism $E\to F$ is a couple $(\phi, \Phi)$, where $\phi:M\to N$ is smooth and $\Phi:\phi^*F\to E$ is a vector bundle morphism. We will denote vector bundle comorphisms by diagrams:
    $$
    \xymatrix{E\ar[d]&F\ar[d]\ar@{-->}[l]\\M\ar[r]&N}
    $$
\end{definition}

\begin{example}\label{ep:morp-comorp}
A morphism of vector bundles  $E\to F$ over $M\to N$, induces a vector bundle comorphism $\xymatrix{E^*&F^*\ar@{-->}[l]}$. In fact, any vector bundle comorphism corresponds to a vector bundle morphism of the dual vector bundles.
\end{example}

\begin{observation}\label{ob:module}
Let $E\to M$ and $F\to N$ be two vector bundles, and $\phi: M\to N$ a smooth map. Then $\Gamma(F)$ is naturally a $C^\infty(N)$-module structure with point-wise multiplication. Moreover, $\Gamma(E)$ is also a $C^\infty(N)$-module, with the product induced by the point-wise multiplication: 
\[
g \cdot \xi := (\phi^*g) \xi, \quad \text{for}\;g\in C^\infty(N), \quad \xi \in \Gamma(E). 
\]
\end{observation}
Then we have the following lemma: 

\begin{lemma}\label{lemma:point-wise}
Let $\Phi: \Gamma(F)\to \Gamma(E)$ be a $C^\infty(N)$-module morphism. Given $\sigma \in \Gamma(F)$ and $x\in M$, we claim that if $\sigma(\phi(x))=0$, then $\Phi(\sigma)(x)=0$. 
\end{lemma}
\begin{proof} As the statement is point-wise, we may take a local frame $F^i$ of $F$ and we set $\sigma = \sum_i f_i F^i$ in a neighborhood of $x$. Then 
\begin{equation}\label{eq:Phi-sigma}
    \Phi(\sigma)=\sum_i \phi^*f_i \Phi(F^i).
\end{equation}
If $\sigma(\phi(x))=0$, then $f_i(\phi(x))=0$. Thus, by \eqref{eq:Phi-sigma}, $\Phi(\sigma)(x)=0$. 
\end{proof}

\begin{cor}\label{cor:module-morp-comorp}
We use the same notation as in the last lemma. Then the set of $C^\infty(N)$-module morphisms $\hom_{C^\infty(N)}(\Gamma(F), \Gamma(E))$ is exactly the set of comorphisms $\xymatrix{E&F\ar@{-->}[l]}$. 
\end{cor}
\begin{proof}
Given $\Phi\in \hom_{C^\infty(N)}(\Gamma(F), \Gamma(E))$, we define $\varphi: \xymatrix{E&F\ar@{-->}[l]}$ by
\begin{equation*}
    \varphi(w, x) := \Phi(\sigma)(x),  \qquad \text{for}\; x\in M, w\in F_{\phi(x)} 
\end{equation*} where $\sigma \in \Gamma(F)$ is a section such that $\sigma(\phi(x))=w$. By Lemma \ref{lemma:point-wise}, this is well defined. It is not hard to verify that such $\Phi$ and $\varphi$ one-to-one correspond to each other.  
\end{proof}

\begin{pro} \label{pro:S-limit} The functor $\mathfrak S : \Z^{<0}\Vectbd\to \NMfd$ preserves limits.
\end{pro}

\begin{proof}
This follows from the fact that $\mathfrak S$ is the right adjoint to a functor $\mathfrak F: \NMfd\to \Z^{<0}\Vectbd$.  We define $\mathfrak F$ on the subcategory of split N-manifolds, which by the N-graded version of the Batchelor-Gawedzki theorem \cite[Eq.~(2) Theorem 1]{BoPo13} is equivalent to the category of all N-manifolds. Let $\huaM$ be a split N-manifold which corresponds to  a negatively graded vector bundle $V_\bullet \to M$, that is $C(\mathcal M)\cong \Gamma(S^{\ge 0}(V_\bullet^*))$. We set $\mathfrak F(\mathcal M)= S^{\ge 0}(V_\bullet)$, with its total grading.  We take another negatively graded vector bundle $W_\bullet\to N$, then by definition, 
\begin{equation}
    \hom(\mathcal M,\mathfrak S(W_\bullet))= \hom_{C^\infty (N)\alg}(\Gamma(S^{\ge 0}(W^*_\bullet)), \Gamma(S^{\ge 0}(V^*_\bullet))). 
\end{equation}
Here we notice that similar to Observation \ref{ob:module}, $\Gamma(S^{\ge 0}(W^*_\bullet)$ and $\Gamma^{\ge 0}(S(V^*_\bullet))$ are both $C^\infty$-algebras with the symmetric product. Similarly, (co)morphisms of vector bundles can also extend to (co)morphisms of vector bundles with additional $\Real$-algebra structure fiber-wise, provided the (co)morphisms preserve the algebra structures. We denote the set of comorphisms by $\cohom(-, -)$, and the one preserving addtional algebra structures by $\cohom_{\alg}(-, -)$. Then Corollary \ref{cor:module-morp-comorp} and Example \ref{ep:morp-comorp} implies that
\begin{equation}
\begin{split}
   & \hom_{C^\infty (N)\alg}(\Gamma(S^{\ge 0}(W^*_\bullet)), \Gamma(S^{\ge 0}(V^*_\bullet)))= \cohom_{\alg}(S^{\ge 0}(W^*_\bullet), S^{\ge 0}(V^*_\bullet)) \\ = & \cohom(W^*_\bullet, S^{\ge 0}(V^*_\bullet))  =  \hom(S^{\ge 0}(V_\bullet), W_\bullet)= \hom(\mathfrak F(\mathcal M),W_\bullet).
\end{split}
\end{equation}

We thus have proven the adjunction formula $\hom(\mathcal M,\mathfrak S(W_\bullet))=\hom(\mathfrak F(\mathcal M),W_\bullet)$. 
\end{proof}

We would like to be able to separate out the highest degree generators from an iterated tangent bundle. I.e., we would like to have diagrams of graded manifolds of the sort:

$$
T[k]M \to T[1]^kM\to \frac{T[1]^kM}{T[k]M}
$$

After choosing splittings, this statement becomes evident: we can just include the highest degree (=k) $TM$ into the larger vector bundle and also realize the quotient as a strict morphism. In fact, the sequence is canonical and does not depend on the choice of a splitting, as follows from the following Lemma:

\begin{lemma} Let $\mathcal M$ be an $N$-manifold with corresponding $\Z^{<0}$-graded vector bundle $V_\bullet$. Then  the following algebras define graded manifolds:  
\begin{itemize}
	\item The smallest subalgebra of $C(\mathcal M)$ containing $C_0(\mathcal M) \oplus C_1(\mathcal M)\oplus...\oplus C_{k-1}(\mathcal M)$---we denote the corresponding N-manifold by $\mathcal M_{< k}$.
	\item The algebra $C(\mathcal M)/I$, where $I$ is the ideal of $C(\mathcal M)$  generated by $C_{1}(\mathcal M)\oplus...\oplus C_{k-1}(\mathcal M)$---we denote the corresponding N-manifold by $\mathcal M_{\geq k}$
\end{itemize}
There are natural maps $\mathcal M_{\geq k} \to \mathcal M \to \mathcal M_{<k}$. A splitting of $\mathcal M$ induces a splitting of $\mathcal M_{\geq k}$ and $\mathcal M_{<k}$, such that this sequence of maps becomes $V_{\geq k}\to V_\bullet \to V_{<k}$.
\end{lemma}

 \begin{proof}
 The proof is a direct computation, which is implicitly carried out when proving the $N$-graded version of the Batchelor-Gawedzki theorem. \cite[Theorem 1]{BoPo13} 
 \end{proof}

Our object of interest $T[k]M$ is naturally split, since it has generators in only one degree. So any distribution $E\subset TM$ induces a graded submanifold $E[k]\to T[k]M$ and the above sequence can be changed to:  
\begin{equation} \label{eq:Ek-split}
    E[k] \to T[1]^kM\to \frac{T[1]^kM}{E[k]}
\end{equation}
Moreover, the associated sequence of graded vector bundles over $M$ splits (non-canonically), so $T[1]^k M\cong \frac{T[1]^kM}{E[k]}\times_M E[k]$.

\subsection{Compatible connections for $\Z$-graded vector bundles}
Given a morphism of $\Z$-graded vector bundles $\Phi:E_\bullet\to F_\bullet$ over manifolds $M$ and $N$, we would like the induced morphism $T[1]\mathfrak S(E_\bullet)\to T[1]\mathfrak S(F_\bullet)$ to also be realized as a morphism of $\Z$-graded vector bundles. I.e., we want to pick connections on $E_\bullet$ and $F_\bullet$ such that $\mathfrak T[1]\mathfrak S(\Phi)$ preserves the corresponding splittings of $T[1]\mathfrak S(E_\bullet)$ and $T[1]\mathfrak S(F_\bullet)$. For this, we introduce the notion of  compatible connections of vector bundle comorphisms.

Vector bundle comorphisms have one important property: They induce morphisms on the level of sections:  Let $\Phi:\xymatrix{E&F\ar@{-->}[l]}$ be a comorphism over $\phi:M\to N$. Then any section $\sigma \in \Gamma(F)$ induces a section $\Phi\circ \phi^*\sigma\in \Gamma(E)$. Recall
that jet bundles are quotients of sections. 
Then a comorphism  $\xymatrix{E&F\ar@{-->}[l]}$ induces the following comorphism
\begin{equation} \label{eq:j1-comor}
    \xymatrix{J^1 \Phi:  J^1E&J^1F\ar@{-->}[l]}, \quad \text{by} \; j^1 \sigma \mapsto j^1(\Phi\circ \phi^*\sigma), 
\end{equation}

The following lemma uses standard techniques (cf. \cite{Sau89}, Chapter 4.2, where the statement is proven for invertible morphisms rather than general comorphisms). 
\begin{lemma}
We use the notion as in the above text. Then the following diagram commutes: 
$$
\xymatrix{T^*M\otimes E\ar[r]& J^1E\ar[r]& E\\
T^*N\otimes F\ar[r]\ar@{-->}[u]& J^1F\ar[r]\ar@{-->}[u]& F\ar@{-->}[u]}
$$
\end{lemma}
\begin{proof}
By definition of a comorphism, we need to use $\phi^*$ to pull back the sequence of vector bundles over $N$ to a sequence of vector bundles over $M$, i.e., the above lemma actually means, that the following diagram of vector bundles over $M$ commutes:
$$
\xymatrix{T^*M\otimes E\ar[r]& J^1E\ar[r]& E\\
\phi^*(T^*N\otimes F)\ar[r]\ar[u]&\phi^*(J^1F)\ar[r]\ar[u]& \phi^*F\ar[u]}
$$

Recall that the natural map $J^1E \to E$ is defined as $(j^1 \sigma)_x \mapsto \sigma_x$ for a section $\sigma \in \Gamma(E) $ and point $x\in M$.  The right square of the diagram commutes by the definition of $J^1 \Phi$ (see Eq.~\eqref{eq:j1-comor}). 

For the left square, let us  describe the proof with the help of local coordinates $y^1,...,y^n$ on $N$ and a local frame $F^j$ on $F$. Then at point $\bar{y}\in N$, the vector bundle morphism 
\begin{equation*}
    (T^*N\otimes F) \to J^1 (F), \quad \text{is defined by}\; \sum_{i,j} a_{i,j}dy^i|_{\bar y} \otimes F^j|_{\bar y} \mapsto (j^1\sigma)_{\bar y}, 
\end{equation*}

with section $\sigma\in \Gamma(F)$  given by $\sigma(y)=\sum_{i,j}a_{i,j}(y^i-\bar y^i)F^j$. \\

The map $\phi^*(J^1F)\to J^1E$ is also best understood in local coordinates. Let us take $x^1,...,x^m$ local coordinates on $M$ and $E^1,...,E^k$ a local frame of $E$. Then a section $\xi$ of $E$ is locally given by $x\mapsto\xi(x)= (x, \sum_{j}\beta_j(x)E^j)$. The first {jet} also contains the first derivatives: 
$$x\mapsto j^1\xi(x)=(x,\sum_{j}\beta_j(x)E^j, \sum_{i,j}\partial_{x^i}\beta_j(x)dx^i\otimes E^j).$$ 
Analogously a section $\sigma$ of $F$ is $y\mapsto \sigma(y)=(y,\sum_{j}\alpha_j(y)F^j)$ and its jet is $$y\mapsto j^1\sigma(y)=(y,\sum_{j}\alpha_j(y)F^j, \sum_{i,j}\partial_{y^i}\alpha_j(y)dy^i\otimes F^j),$$ where $y^1, \dots, y^n$ are local coordinates on $N$ and $F^1, \dots, F^l$ are a local frame of $F$. Then a pullback section $\phi^* \sigma \in \Gamma(\phi^* F)$ is $x\mapsto \phi^* \sigma(x)=(x,\sum_{j}\alpha_j(\phi(x))F^j)$ and its jet is $$x\mapsto j^1 \phi^* \sigma(x)=(x,\sum_{j}\alpha_j(\phi(x))F^j, \sum_{i,j}\partial_{x^i}\alpha_j(\phi(x))dx^i\otimes F^j).$$ 
The expression $\Phi\circ \phi^*\sigma \in \Gamma(E)$ takes the form,
$$
x\mapsto \Phi\circ \phi^*\sigma(x)=(x, {\sum_{i,j}\Phi_{j}^i(x)\alpha_i(\phi(x))E^j}),
$$
where in local coordinates, $\phi=(\phi^1,...,\phi^n)$ is a collections of maps depending on $x^1,...,x^m$, and $\Phi=\sum_{i,j}\Phi_{j}^i(x)F_i^*\otimes E^j$ is a matrix depending on a point in $M$, with $F_i^*$'s the dual frame of $F^i$'s. 
We now can calculate $ j^1(\Phi\circ \phi^*\sigma)$:
$$ j^1(\Phi\circ \phi^*\sigma)(x) = \left(x, {\sum_{i,j}\Phi_{j}^i(x)\alpha_i(\phi(x))E^j}, 
{\sum_{i,j,p}\left(
\partial_p\Phi_{j}^i(x)\alpha_i(\phi(x))+
\sum_q\Phi_{j}^i(x) (\partial_{y^q}\alpha_i)(\partial_{x^p}\phi^q(x))
\right) dx^p\otimes E^j
}
\right).
$$
In particular, we know that if we denote the derivative coordinates in $J^1F$ by $u_{i,q}$,  that is, $u_{i, q}= \partial_{y^q} \alpha_i(y)$, the map $\phi^*(J^1F)\to J^1E$ becomes:
$$
(x,\sum_i \alpha_iF^i, \sum_{i,q}u_{i,q}dy^q\otimes F^i) \mapsto  \left(x, {\sum_{i,j}\Phi_{j}^i(x)\alpha_iE^j}, 
{\sum_{i,j,p}\left(
\partial_p\Phi_{j}^i(x)\alpha_i+
\sum_q\Phi_{j}^i(x) u_{i,q}\partial_{x^p}\phi^q(x)
\right) dx^p\otimes E^j
}
\right).
$$ 
In  particular for elements of the form $(x,  0, dy^q\otimes u_{i,q} F_i) \in \phi^*(T^*N \otimes F)\subset  \phi^*J^1F$,  since $\alpha_i(x)=0$ for these elements, we have
$$
\sum_{i,j,p}\left(
\partial_{x^p}\Phi_{j}^i(x)\alpha_i+
\sum_q\Phi_{j}^i(x) u_{i,q}\partial_{x^p} \phi^q (x)
\right) dx^p\otimes E^j$$ $$= \sum_{i,j,p}
\sum_q\Phi_{j}^i(x) u_{i,q}\partial_p\phi^q(x)
dx^p\otimes E^j= ( \phi^* \otimes \Phi) \left(\sum_{i,q}dy^q\otimes u_{i,q} F^i\right),
$$
which is exactly in the image of $\phi^*(T^*N \otimes F) \to T^*M \otimes E$. Thus, the left square commutes.   
\end{proof}

\begin{definition}Let $\xymatrix{E & \ar@{<--}[l] F}$ be a vector bundle comorphism over $M\to N$ and $s:E\to J^1E$ and $s':F\to J^1F$ be connections. We say that $s,s'$ are compatible if
the following diagram commutes
$$\xymatrix{
J^1E&\ar[l]_s E\\ J^1F\ar@{-->}[u]&\ar[l]_{s'}F\ar@{-->}[u].
}$$ This definition generalizes easily to the graded case by taking levelwise compatible connections.
\end{definition}

\begin{example}
Let $\phi:M\to N$ be smooth and $F\to N$ a vector bundle. Then we have the following diagram of vector bundle morphisms:
$$
\xymatrix{\phi^*(T^*N\otimes F)\ar[r]\ar[d]& \phi^*J^1F\ar[r]\ar[d]& \phi^*F\\
T^*M\otimes \phi^*F\ar[r]&J^1(\phi^*F)\ar[r]& \phi^*F\ar@{=}[u]}
$$
A section $s:F\to J^1F$ induces a section $\phi^*s:\phi^*F\to \phi^*J^1F$, which in turn induces a unique connection $s':\phi^*F\to J^1\phi^*F$, which is compatible with $s$. Abusively we denote this \emph{pullback connection} $s'$ by $\phi^*s$ and note that this coincides with the classical definition of pullback connection. 
\end{example}

\begin{example}
Given a connection $s$ on $F\to N$, there is a unique connection $s^*$ on $F^*\to N$ such that the induced connection on $F\otimes F^*$ is compatible with the trivial connection on $ N\times \mathbb R$ via the trace comorphism $\xymatrix{N\times \mathbb R&F\otimes F^*\ar@{-->}[l]}$. One can verify that the dualization is compatible with pullbacks, i.e., $\phi^*(s^*)=(\phi^*s)^*$ as connections on $\phi^*F^*$ (cf. \cite{tuchar}). 
\end{example}

The reason why we are interested in compatible connections is the following result, which may be viewed as the functorial version of Theorem \ref{thm:splitgrad}. 
\begin{thm}
Let $E_\bullet \xrightarrow{L} F_\bullet$ be a morphism of $\Z$-graded vector bundles over $M\xrightarrow{\phi} N$ and $s,s'$ compatible connections on the corresponding comorphism between $E_\bullet^* , F_\bullet^*$. Then the morphism of graded manifolds $T[1]\mathfrak S(E_\bullet)\xrightarrow{T[1] \mathfrak S (L)} T[1]\mathfrak S(F_\bullet)$ respects the splittings induced by $s,s'$ given in Theorem \ref{thm:splitgrad}.
\end{thm}
\begin{proof}
The theorem amounts to proving the following commutative diagram:
\begin{equation*}
    \xymatrix{T[1]\mathfrak{S}(E_\bullet) \ar[r]^{T[1]\mathfrak S (L) } \ar[d]^{sp_{s'}} \mathfrak{S}(L) & T[1] \mathfrak S (F_\bullet) \ar[d]^{sp_{s'}}\\ 
    E_\bullet \oplus T[1]M \oplus  E[1]_{\bullet} \ar[r]^{L\oplus T[1]\phi \oplus L[1]}& F_\bullet \oplus T[1]M \oplus  F[1]_{\bullet}},
\end{equation*}
where $sp_{s}, sp_{s'}$ denotes the splittings induced by $s, s'$ respectively in Theorem \ref{thm:splitgrad}. 
Then the proof amounts to checking what happens with a $\Z$-graded vector bundle morphism when following the construction of Theorem \ref{thm:splitgrad}.  We look at $T[1]\mathfrak S (L)$ on $\mathcal T$-points for a test $\Z$-graded manifold $\huaT$, i.e., the map,
$$
\hom(\Gamma(S^\bullet(E^*_\bullet)), C(\mathcal T)\otimes \mathbb R[\epsilon]) \to 
\hom(\Gamma(S^\bullet(F^*_\bullet)), C(\mathcal T)\otimes \mathbb R[\epsilon]) .
$$
Following the notation in Theorem \ref{thm:splitgrad}, an element on the left can be described as $\Psi_0 + \epsilon\Psi_{1}$. The image of such an element is given by $\Psi_0\circ L^* + \epsilon\cdot \Psi_{1}\circ L^*=:\Phi_0+\epsilon \cdot \Phi_{1}$, where $L^*$ denotes the map  $L^*:\Gamma(S^\bullet(F^*_\bullet))\to \Gamma(S^\bullet(E^*_\bullet))$ induced by (the symmetric powers of the dual of) $L$. \\

The fact that $\Phi_0=\Psi_0\circ L^*$ immediately implies that the component of $T[1]\mathfrak S(L)$ on the $E_\bullet\to F_\bullet$ part is just precomposition with $L^*$, therefore given by $L$ in the graded vector bundle picture. Now we have to apply the jet construction of Theorem \ref{thm:splitgrad} to $\Phi_{1}=\Psi_{1}\circ L^*$. The restriction of $\Phi_{1}$ to $\Gamma(F_\bullet^*)$ can be described as follows:
$$
\Gamma(F_\bullet^*)\xrightarrow{L^*} \Gamma(E_\bullet^*)\xrightarrow{\Psi_{1}} C(\mathcal T)
$$
This diagram induces a diagram of jet bundles:
$$
\Gamma(J^1F_\bullet^*)\overset{J^1L^*}{\to} \Gamma(J^1E_\bullet^*)\overset{J^1{\Psi_{1}}}{\to} C(\mathcal T).
$$
We are now interested in the composition $J^1\Phi_{1}\circ s' = J^1\Psi_{1}\circ J^1L^*\circ s'= J^1\Psi_{1}\circ s\circ L^*$, where the last equality follows from the compatibility of the connections. In particular, the component of $T[1]\mathfrak S(L)$ on the $E[1]_{\bullet}\to F[1]_{\bullet}$ part is just composition with $L^*$, thus is given by $L[1]$ in the graded vector bundle picture. With an  analogous computation,   the component of  $T[1]\mathfrak S(L)$ on the $T[1]M\to T[1]N$ part coincides with the tangent map $T[1]\phi: T[1]M\to T[1]N$ (without a choice of connection).
\end{proof}

A particular case where compatible connections can always be found, is the case of fiberwise surjective vector bundle morphisms.

\begin{lemma}
Let $E\to F$ be a fiberwise surjective vector bundle morphism and $s$ a connection on $F^*$. Then there exists a compatible connection on $E^*$.
\end{lemma}

\begin{proof}
The map $\phi^*F^*\to E^*$ is injective, i.e., $E^*\cong \phi^*F^*\oplus K$ for some vector bundle $K$. Pick as the connection on $E^*$ the direct sum $s\oplus s^K$, where $s^K$ is any connection on $K$.
\end{proof}

\begin{cor}\label{cor:comp-conn}
Let $M\to N$ be a submersion. Then for any connection on $T^*N$, there exists a compatible connection on $T^*M$. In particular, the morphisms $T[1]^kM\to T[1]^kN$ can all be realized by strict morphisms (i.e., coming from morphisms of graded vector bundles).
\end{cor}

\section{Representability of the tangent functor}

\begin{definition}\label{defi:T} \cite{severa:diff} The tangent functor $\mathcal T: \simp{\Mfd}\to \pSh{\Z\Mfd}$ from the category of simplicial manifolds $\simp{\Mfd}$ to that of presheaves of  $\mathbb Z$-graded manifolds $\pSh{\Z\Mfd}$ is given by, 
\begin{equation*}
  \huaT(X_\bullet):= \Hom( D_\bullet, X_\bullet): T\mapsto \hom(T\times D_\bullet, X_\bullet), \quad \forall X_\bullet \in \simp{\Mfd}, \forall T\in \Z\Mfd,
\end{equation*} on the level of objects, and by post-composition on the level of morphisms. 
Here $T$ denotes also the constant simplicial manifold with $T$ on each level and the $\hom$ is taken in the category $\simp{(\Z\Mfd)}$ of simplicial $\Z$-graded manifolds. 
\end{definition}
\begin{remark}
The space $\Hom(D_\bullet, X_\bullet)(T)$ consists of maps in $ \hom(T\times D_l, X_l)$ which are compatible with each other with respect to face and degeneracy morphisms.  
Therefore $\Hom(D_\bullet, X_\bullet)$ is the limit of the following diagrams
\[
\xymatrix{
\Hom(D_k, X_k) \ar[r]^{(d^D_J)^*} \ar[d]_{(s^X_I)_*} & \Hom(D_{k+1}, X_k)^{\times [k+1]} \\
\Hom(D_k, X_{k+1})^{\times [k]} & \Hom(D_{k+1}, X_{k+1})\ar[l]^{(s^D_I)^*} \ar[u]^{(d^X_J)_*}
}
\]
where $I=\{0, \dots, k\}$ and $J=\{0, \dots, k+1\}$. 
\end{remark}

The goal of this section is to prove the following central theorem.
\begin{thm}\label{thm:ker-p}
Let $X_\bullet$ be a Lie $n$-groupoid. The presheaf $\mathcal T(X_\bullet)$ is an $N$-manifold, that is, it is representable in $\NMfd$.
Moreover, it is (non-canonically) isomorphic to the graded manifold induced by
the tangent object, 
\begin{equation}\label{eq:tan-cx}
 \mathbb T(X_\bullet):= \ker Tp^1_0|_{X_0} [1] \oplus \ker Tp^2_0|_{X_0} [2]  \oplus \dots \oplus \ker Tp^n_0|_{X_0} [n], 
\end{equation}
where $p^k_0: X_k \to \Horn{k}{0}(X)$ is the horn projection. 
\end{thm}
%The theorem is a consequence of the results in the below sections. 
%We start by proving that $\mathcal T(X_\bullet)$ is the limit of certain presheaves $H^k$ (Lemma \ref{lim}). Then we show the representability  of these presheaves inductively (Lemma \ref{lem:induction}).

The rest of this section is devoted to proving Theorem \ref{thm:ker-p}. We will first show that $\mathcal T(X)$ can be constructed as the limit of certain simpler presheaves $H^k$ (Lemma \ref{lim}). We then show that $H^k$ can be constructed as a certain pullback inductively from $H^{k-1}$ (Lemmas \ref{lem:case0}, \ref{lem:hkaspullb}). To show that $H^{k}$ is representable, we will need an intermediate space $A\times_C B$. We will first analyze its components $A$ and $B$ separately (Lemmas \ref{lem:key}, \ref{lem:B}), before looking at $A\times_C B$ altogether (Lemmas \ref{lem:ss}, \ref{lem:abc}) and concluding the induction (Lemma \ref{lem:ind-step}). This then implies the representability of $H^k$ asserted in Lemma \ref{lem:induction} and hence the representability of $\mathcal T(X)$. The rough proof structure is indicated in the below diagram.\\

\tikzstyle{process} = [rectangle, minimum width=3cm, minimum height=1cm, text centered, draw=black]
\tikzstyle{arrow} = [thick,->,>=stealth]
\begin{tikzpicture}[node distance=1.5cm]
\node (start)[process, align=center]{Lemma \ref{lim}: $\mathcal T$ as a limit of $H^k$};
\node (induction)[process, below of = start]{Lemma \ref{lem:induction}: Calculation of $H^k$ by induction};
\node (inductionstep)[process,below of = induction, xshift=-4cm]{Lemma \ref{lem:ind-step}: Induction step $H^{k-1}\mapsto H^{k}$};
\node (inductionstart)[process, right of = inductionstep, xshift=6cm]{Lemma \ref{lem:case0}: Initialization of induction, $H^0$};
\node (hkaspullback)[process,below of = inductionstep]{Lemma \ref{lem:hkaspullb}: $H^k$ as a pullback of $H^{k-1}$};
\node (intermediatespace)[process,right of = hkaspullback, xshift = 6.5cm]{Lemmas \ref{lem:ss},  \ref{lem:abc}:  Analysis of intermediate space $A\times_CB$};
\node (key)[process,below of = intermediatespace, xshift = -7cm]{Lemma \ref{lem:key} (key technique): Characterization of $A$};
\node (lemB)[process,right of = key, xshift = 7cm]{Lemma \ref{lem:B}: Characterization of $B$};

\draw [arrow] (induction) -- (start);
\draw [arrow] (inductionstart) -- (induction);
\draw [arrow] (inductionstep) -- (induction);
\draw [arrow] (hkaspullback) -- (inductionstep);
\draw [arrow] (intermediatespace) -- (inductionstep);
\draw [arrow] (key) -- (intermediatespace);
\draw [arrow] (lemB) -- (intermediatespace);
\end{tikzpicture}
Lemma~3.9 is the key technical input used to repair the combinatorial gap mentioned in the introduction.
\begin{remark}
 The non-canonical isomorphism in Theorem \ref{thm:ker-p} of the infinitesimal data of a Lie $n$-groupoid $X_\bullet$ to the tangent object \eqref{eq:tan-cx} depends on auxiliary choices: namely a choice of compatible connections on $TX_k$ and $T\Horn{k}{0}(X)$, together with the isomorphism \eqref{eq:hk-iso}, for $k=2, \dots, n$. The connections on different levels can be independent. We need these data even when $X_0$ is a point.
 
By contrast, the differentiation of Lie groups and Lie 2-groups to Lie algebras and Lie 
2-algebras is canonical and does not involve any choice of connection. Accordingly, in the pointed case, one expects the dependence on connections to disappear and the identification in Theorem~\ref{thm:ker-p} to admit a canonical form. This expectation is confirmed in recent work of Rogers~\cite{rogers:2025} for Lie $\infty$-groups.

For general base manifolds $X_0$ and $n\ge 2$, however, we expect connection data to remain unavoidable. Conceptually, passing to an 
N-manifold requires the choice of a splitting, and such a splitting depends on choices of splittings of certain short exact sequences at each level $\ge 2$; see \cite[Eq.~(2) in Theorem~1]{BoPo13}.  Our contribution is, at minimum, to make these choices explicitly, thereby rendering the resulting isomorphism in Theorem~\ref{thm:ker-p} concrete and geometric, even though it is not canonical. 
\end{remark}

We can now state the following Corollary:
\begin{cor} \label{cor:Q-structure}
The tangent object $\mathbb T(X_\bullet)$ of a Lie $n$-groupoid $X_\bullet$ carries the structure of a Lie $n$-algebroid. 
\end{cor}

\begin{proof}
Lie $n$-algebroid structures on $\mathbb T(X_\bullet)$ are in one-to-one correspondence with homological vector fields (i.e. degree one derivations on $\mathcal C(\mathfrak S(\mathbb T(X_\bullet)))$ squaring to zero) on the graded manifold $\mathfrak S(\mathbb T(X_\bullet))$. These in turn coincide with $D$-actions on $\mathfrak S(\mathbb T(X_\bullet))$ \cite[Lemma 9.1]{severa:diff}. Since Theorem \ref{thm:ker-p} provides an isomorphism between $\mathfrak S(\mathbb T(X_\bullet))$ and $\mathcal T(X_\bullet)$, we only need to provide a $D$-action on the presheaf $\mathcal T(X_\bullet)=\Hom(D_\bullet, X_\bullet)$. Such an action can be produced by observing that there is a natural inclusion of monoids
$$D\hookrightarrow T[1]D=\Hom(D,D)\hookrightarrow \Hom(D_\bullet,D_\bullet)$$
and using the natural action of $\Hom(D_\bullet,D_\bullet)$ on $\Hom(D_\bullet, X_\bullet)$ by precomposition.
\end{proof}

\subsection{Description of $\mathcal T(X_\bullet)$ as a limit} \label{sec:H1-4}

We introduce a sequence  of presheaves $H^k:\Z\Mfd\to \Set$:
\begin{align}
    H^k(T)=\Big\{(f_0,...,f_{k})~|~ &f_l\in \hom(T\times  D_l,X_l) \mathrm{~for~} l<k, \\  & f_k\in \hom(T\times d_{1,...,k}^{-1}(\pt), X_k)
     ~{such~ that~conditions~H1.-H4.~hold} \Big\}  \nonumber
\end{align}
Here $d_{1...k}:D_{k}\to D_{0}$ is the face map corresponding to the inclusion $[0]=\{0\}\to [k]=\{0,...,k\}$ and $\pt$ is the basepoint of $D_0=D$, and $\pt \to D$ corresponds to $\R[\epsilon]\to \R$. This gives us,  
\begin{equation} \label{eq:non-embed}
    d_{1,...,k}^{-1}(\pt)=\pt \times D^{\times k} \xrightarrow{\iota_k} D^{\times k+1} = D_k, 
\end{equation} which is in general not an ``inclusion'' in the usual sense.
\begin{enumerate}[label=\upshape(H\arabic*),ref=
H\arabic*]
    \item\label{itm:slower} $s_i^{X}f_l=f_{l+1}s_i^{ D}$ for $l\leq k-2$ and $i \in \{ 0,...,l \}$.
    \item\label{itm:dlower} $d_i^{X}f_{l+1}=f_{l}d_i^{ D}$ for $l\leq k-2$ and $i \in \{ 0,...,l+1 \}$.
    \item\label{itm:skk-1} $s_i^{X}f_{k-1}|_{d^{-1}_{1,...,k-1}(\pt)}=f_{k}s_i^{ D}|_{d^{-1}_{1,...,k-1}(\pt)}$ for $i \in 0,...,k-1$.
    \item\label{itm:dkk-1} $d_i^{X}f_{k}=f_{k-1}d_i^{ D}|_{d^{-1}_{1,...,k}(\pt)}$ for $i \in 0,...,k$.
\end{enumerate} 
Here $s^X_i$, $d^X_i$ denote the simplicial morphisms for $X_\bullet$, and $s^D_i$, $d^D_i$ denote those for $D_\bullet$. Moreover, $|_{d^{-1}_{1,...,k-1}(\pt)}$ and $|_{d^{-1}_{1,...,k}(\pt)}$ denote the natural precomposition with $\iota_{k-1}$ and $\iota_k$ respectively.

\begin{lem}\label{lim} Let $X_\bullet$ be any simplicial manifold and $H^k$ the presheaves defined above. Then $\mathcal T(X_\bullet)=\lim H^k$. 
\end{lem}
\begin{proof}
There are natural maps 
$$H\to ... \to H^{k+1}\to H^k \to ....\to H^0$$
On $T$-points it is just given by restricting to a sub-collection of the $\{f_i\}$. So we have a canonical map $H\to \lim H^k$. In order to prove the equality, we have to show that there is an inverse map $\lim H^k\to H$. We define this map on $T$-points. \\

An element in $\lim H^k(T)$ is represented by a series of elements $(f^{(0)}_0)\in H^0(T)$,  $(f^{(0)}_0,f^{(0)}_1)\in H^1(T)$, ... , $(f^{(i)}_0,...,f^{(i)}_i)\in H^i(T)$,... 
By definition, all these maps have to be compatible with restrictions and simplicial identities whenever defined. In particular,  $f^{(i)}_a=f^{(j)}_a$ whenever $i,j\geq a+1$. We can hence define the simplical morphism $F\in \hom(T\times  D_\bullet, X_\bullet)$ by $F_{a}=f^{(a+1)}_a\in hom(T\times  D_a, X_a)$. The simplical identities for $F$ between levels $a$ and $a+1$ then follow from $F_a=f^{(a+1)}_a=f^{(a+2)}_a$, $F_{a+1}=f^{(a+2)}_{a+1}$ and the fact that $f^{(a+2)}_{a}, f^{(a+2)}_{a+1}$ satisfy all simplical identities.
\end{proof}

For $H^0$ all conditions are void and $d_{1, \dots, k}$ should be understood as $d_{\emptyset}:D\to D$, and is the identity. Therefore $d^{-1}_{\emptyset}(\pt)=\pt$. This means that:

\begin{lem}\label{lem:case0}
$H^0(T)=\hom(T,X_0)$. Thus, $H^0=X_0$ is representable. 
\end{lem}

\subsection{The central pullback diagram} \label{sec:pullback-diag}
In this subsection, we will  build $H^k$ from $H^{k-1}$ via a pullback diagram.

\begin{lemma}\label{lem:hkaspullb}
$H^k(T)$ can be realized as the following pullback: 
\begin{equation} \label{diag:hk-combi}
    \xymatrixcolsep{15pc}\xymatrix{H^k(T) \ar[r]^{\huaF_k} \ar[d] & \hom(T\times \pt \times D^k, X_k)  \ar[d]_{(s^D_0)^*, \dots, (s^D_{k-1})^*,  (d_1^X)_*, \dots, (d_k^X)_*} \\
    H^{k-1}(T) \ar[r]^-{((s_0^X)_*, \dots, (s_{k-1}^X)_*, (d_1^D)^*, \dots, (d_k^D)^*) \circ \huaF_{k-1}} &  \hom(T\times \pt \times D^{k-1}, X_k)^{[k-1]}  \times \hom (T\times \pt \times D^k, X_{k-1})^{[k]\backslash 0}  }.
\end{equation}
Here, for $l\ge 1$,  $\huaF_{l}$ is the forgetful map that remembers only the information on the $l$-th level restricted on $\pt \times D^{l}$, that is 
\begin{equation}
    H^{l} \xrightarrow{\huaF_{l}} \Hom(\pt \times D^{l}, X_{l}). 
\end{equation}
\end{lemma}

\begin{proof}
We start by observing that $(f_0,...,f_k)\in H^k(T)$ is uniquely determined by $f_k$ and then determine the conditions on $f_k$ explicitly. \\

\begin{enumerate}[label=\upshape(Step\arabic*),ref=
Step\arabic*]
\item \label{step:determined-by-fk} {\bf An element in $(f_0,...,f_k)\in H^k(T)$ is uniquely determined by $f_k$.}\\
    
    For $l<k$ consider the iterated inclusion $I_{l, k}$ from $ D_l=D^{[l]}=D^{l+1}$ to $D_k$ by augmenting ``points'' from the left, i.e., $D^{l+1}\to {\pt}^{k-l}\times D^{l+1}\hookrightarrow D^{k+1}$. 
    We recall that $d^{0,\dots,k-l-1}: [l]\to [k]$ is a map forgetting $0, \dots, k-l-1$, that is, $i\mapsto k-l+i$ for $i=0, \dots, l$. 
    Thus, $d^D_{0, \dots, k-l-1} \circ I_{l, k} = \id$. Abusively, we write $I_{l, k}$ for $(id_T\times I_{l, k})$.
    The image of $I_{l, k}$ lies in $\pt \times D^k=d_{1...k}^{-1}(\pt)$, so
 $f_k\circ I_{l, k}$ is well-defined and we can deduce 
 \begin{align}\label{eq:lowerdef}f_l= f_l\circ d^D_{0,...,k-l-1}\circ  I_{l, k}=d^X_{0,...,k-l-1}\circ f_k\circ I_{l,k}.\end{align} Here we have used the fact that $f_i$'s are compatible with degeneracies, and one can verify that this equation holds even for the case $l=k-1$.  Hence, $f_l$ is uniquely determined by $f_k$ for $l=0, \dots, k-1$. \\

 \item \label{step:sd-eq-fk} {\bf Let $f^{(k-1)}=(f^{(k-1)}_{0}, ...,f^{(k-1)}_{k-1})\in H^{k-1}(T)$. An element $f_{k}\in \hom(T\times d_{1,...,k}^{-1}(\pt), X_k)$ generates an element of $H^k(T)$, which restricts to $f^{(k-1)}$ if and only if   \begin{itemize}
    \item[(sIk)] $s_i^{X}f_{k-1}^{(k-1)}=f_{k}s_i^{ D}|_{d^{-1}_{1,...,k-1}(\pt)}$ for $i \in I=\{ 0,...,k-1\}$.
    \item[(dJk)] $d_i^{X}f_{k}=f^{(k-1)}_{k-1}d_i^{ D}|_{d^{-1}_{1,...,k}(\pt)}$ for $i \in J=\{1,...,k\}$.
 \end{itemize}}
 
 That is, given an element $f_k \in \hom(T\times d_{1,...,k}^{-1}(\pt), X_k)$, it extends to an element $f=(f_0, \dots, f_k) \in H^k$ with property $f_i=f^{(k-1)}_i$ ($i=0, \dots, k-2$) and $f_{k-1}\iota_{k-1} = f^{(k-1)}_{k-1}$ for an $f^{(k-1)}=(f_0^{(k-1)}, \dots, f_0^{(k-1)}) \in H^{k-1}(T)$ if and only if (sIk) and (dJk) hold.

The necessity of the conditions is obvious because conditions \ref{itm:skk-1} and \ref{itm:dkk-1}
for $f$ imply (sIk), (dJk) respectively.  Notice that it is not a typo that $i\neq 0 $ in (dJk). 

For the sufficiency, we define $\{f_l\}_{l<k}$ using formula \eqref{eq:lowerdef} and obtain $f=(f_0, \dots, f_k)$. We need to show that $f$ satisfies \ref{itm:slower}, \ref{itm:dlower}, \ref{itm:skk-1}, \ref{itm:dkk-1}, 
and $f_{k-1}\iota_{k-1} = f^{(k-1)}_{k-1}$ for $f^{(k-1)}=(f_0^{(k-1)}, \dots, f_{k-1}^{(k-1)}) \in H^{k-1}(T)$. 

The $I_{l, k}$ we have used in the previous step can be rewritten as $I_{l, k}=\iota_k\circ \iota_{k-1}\circ .... \circ \iota_{l+1}$.
Notice that $\iota_{k}=s^D_0$, by \eqref{eq:lowerdef} and (s0k),  we calculate:
 $$f_{k-1}\iota_{k-1}=d^X_0f_{k} \iota_{k} \iota_{k-1}=d^X_0 f_k s^D_0 \iota_{k-1}=d^X_0s^X_0f_{k-1}^{(k-1)}\iota_{k-1}=f_{k-1}^{(k-1)}\iota_{k-1}. $$
As  $\iota_{k-1}:D^{k-1}\to d_{1,...,k-1}^{-1}(\pt)$ is an isomorphism, it follows that $f_{k-1}|_{d^{-1}_{1.,,,.k-1}(\pt)}=f_{k-1}^{(k-1)}$, i.e., $f_{k-1}\iota_{k-1}=f_{k-1}^{(k-1)}$.

Let us look at the analogue for $k-1$ of formula \eqref{eq:lowerdef}, stating that $f_{k-1}^{(k-1)}$ completely determines $f_i^{(k-1)}$ for $i\leq k-1$: $$f^{(k-1)}_l=d_{0,...,(k-1)-l-1}^X\circ f^{(k-1)}_{k-1}\circ I_{l,k-1}$$
Since the image of $ I_{l,k-1}$ lies in the image of $\iota_{k-1}$, this is equal to:
$$
=d_{0,...,(k-1)-l-1}^X\circ f_{k-1}\circ I_{l,k-1}=d_{0,...,(k-1)-l-1}^X\circ d_{0}^X\circ f_{k}\circ I_{1,k}\circ I_{l,k-1}=f_{l}
$$
In particular, this implies that all simplicial identities for $\{f_0,...,f_{k-2}\}$ hold.
Next we check the simplicial identities involving $f_{k-1}$ and $f_{k-2}$: \\
To see the relation $s^X_a f_{k-2}=f_{k-1}s^D_a$, for $a=0, \dots, k-2$,  we can apply formula \eqref{eq:lowerdef} to the relation between $f_{k-1}$ and $f_{k-2}$
$$s^X_a f_{k-2}=s^X_a d^X_0 f_{k-1} \iota_{k-1}=d^X_0s^X_{a+1}f_{k-1}\iota_{k-1}=d^X_0 f_k s^D_{a+1} \iota_{k-1}=d^X_0f_k \iota_k s^D_a=f_{k-1}s^D_a.$$
Similarly $d_a^Xf_{k-1}=f_{k-2}d^D_a$ for $a\in \{0,...,k-1\}$ can be verified by:
$$d^X_af_{k-1}=d^X_ad_0f_{k}I_k=d^X_0d_{a+1}^Xf_{k}I_k=d_0^Xf_{k-1}d^D_{a+1}I_k=d_0^Xf_{k-1}I_{k-1}d^D_a=f_{k-2}d^D_a.$$ 
Finally, almost all the identities between $f_k$ and $f_{k-1}$ follow directly from (sIk) and (dJk), for instance, for \ref{itm:dkk-1} ($i\neq 0$) we have
$$d_i^Xf_k\iota_k=f^{(k-1)}_{k-1}d_i^D\iota_k=f^{(k-1)}_{k-1}\iota_kd_{i-1}^D=f_{k-1}\iota_kd_{i-1}^D=f_{k-1}d_i^D\iota_k.$$
The only special identity is $d_0^{X}f_{k}=f_{k-1}d_0^{ D}|_{d^{-1}_{1,...,k}(\pt)}$, which we can verify by applying from the right $d_0^D$ to the equation defining $f_{k-1}$. 
$$
f_{k-1}\circ d_0^D=d_0^X\circ f_k\circ \iota_k\circ d_0^D
$$
The statement now follows from $\iota_k\circ d_0^D|_{\pt \times D^k}=id_{\pt \times D^k}$.
\end{enumerate}
\end{proof}

Since the presheaves on the right-hand side are representable, we could rewrite the diagram as:

\begin{equation} \label{diag:hk}
    \xymatrixcolsep{15pc}\xymatrix{H^k \ar[r]^{\huaF_k} \ar[d] & T[1]^{k}X_k  \ar[d]_{(s_I^D)^*, (d_J^X)_* } \\
    H^{k-1} \ar[r]^-{((s_I^X)_*, (d_J^D)^*) \circ \huaF_{k-1}} & (T[1]^{k-1}X_k)^{[k-1]}   \times  (T[1]^{k} X_{k-1})^{[k]\backslash 0}  },
\end{equation}

where $J=(1, \dots, k)$ and $I=(0, \dots, k-1)$, and we furthermore abbreviate the tuple $((d^X_1)_*, \dots, (d^X_k)_*) $ to $(d^X_J)_*$, and similarly for any other such tuple of maps. 

If $H^{k-1}$ is representable, showing the representability of $H^k$ is equivalent to showing that there is an N-manifold fitting into the pullback diagram \eqref{diag:hk}. More precisely, we want to inductively prove

\begin{lemma}\label{lem:induction}
For all $k\ge 0$,
 $H^k$  is representable by $\oplus_{i=1}^k \ker Tp^i_0[i]|_{X_0}$. 
\end{lemma}

In order to do so, we first have to get a better understanding of the image of $T[1]^{k}X_k $ in $(T[1]^{k} X_{k-1})^{[k]\backslash 0}  \times (T[1]^{k-1}X_k)^{[k-1]} $ (for notations see Section \ref{sec:iterated-tan}). We will first investigate the image of $(s^D_I)^*$, then $(d^X_J)_*$ and then combine them.  But before this, let us verify the initial cases for the induction. The case when $k=0$ is verified in Lemma \ref{lem:case0}. We now verify that $H^1=\ker T[1]p^1_0|_{X_0}$, and $\huaF_1$ is an embedding. 

\subsubsection{The case when $k=1$}\label{sec:k=1} 
When $k=1$, conditions \eqref{itm:slower} \eqref{itm:dlower} are empty, and \eqref{itm:skk-1} and \eqref{itm:dkk-1} are 
\begin{enumerate}
    \item[(H3)] $s_0^X f_0|_\pt = f_1 s^D_0|_\pt$;
    \item[(H4)] $d_j^X f_1 = f_0 d_j^D|_{\pt \times D}, \quad j=0, 1.$
\end{enumerate}
We have $$H^1(T) = \{ (f_0, f_1) | f_0 \in \hom(T\times D, X_0), f_1\in \hom(T\times \pt \times D, X_1), \;\text{such that H3 and H4 hold.\}}$$ By \eqref{step:determined-by-fk}, $(f_0, f_1)$ is uniquely determined by $f_1$. By \eqref{step:sd-eq-fk}, given a $T$-point $f^{(0)}_0 \in H^0(T) $,   $f_1 \in \hom(T\times \pt \times D, X_1)$ is anything satisfying 
\begin{equation} \label{eq:f0-f1}
    s^X_0 f^{(0)}_0 = f_1 s^D_0|_\pt, \quad d^X_1 f_1 = f^{(0)}_0 d^D_1|_{\pt \times D} = f^{(0)}_0(\pt) = {x_0}. 
\end{equation} As $H^0=X_0$ and $  \hom(T\times \pt \times D, X_1)= T[1]X_1(T)$ are both representable, the above equations are also satisfied by the points they correspond to.  That is, $f_0^{(0)}$ corresponds to a point $x_0 \in H^0=X_0$, and $f_1$ corresponds to a vector $v \in T[1]X_1$. The first equation in \eqref{eq:f0-f1} implies that $f_1|_{(\pt, \pt)}=s^X_0(x_0) $. This shows that $v \in T_{x_0} [1]X_1$. The second equation in \eqref{eq:f0-f1} shows that $Td^X_1 v = 0_{x_0}$. Thus, $H^1 = \ker T[1] p^1_0|_{X_0}$. 

In this case, $\huaF_1$ is an embedding. Noticing that $H^0=X_0\xrightarrow{\huaF_0=id} X_0$, the fact $H^0=X_0 \xrightarrow{(\delta_1, (s^X_0)_*)} T[1]X_0 \times X_1$ is an embedding implies that $\huaF_1$ is an embedding via the pullback diagram \eqref{diag:hk}. However, we should not expect that $\huaF_k$ in general to be a strict embedding. This will be addressed in a later work.

\subsection{Embedding  $T[1]^kM/T[k]M$ into $(T[1]^{k-1}M)^{[k-1]}$}\label{sec:key-A}

In this section we show that for any manifold $M$, the N-manifold $T[1]^kM/T[k]M$ can be seen as an N-submanifold of  $T[1]^{k-1}M^{[k-1]}$. More precisely, for a fixed $k$, consider following diagrams for all $0\leq i\leq j\leq k-2$:
\begin{equation}\label{diag:A}
    \xymatrix{
T[1]^kM/T[k]M\ar[d]^{\sigma_{j+1}} \ar[r]^{\sigma_i} & T[1]^{k-1}M\ar[d]^{\sigma_j
}\\
T[1]^{k-1}M\ar[r]^{\sigma_{i}} & T^{k-2}[1]M.
}
\end{equation}
Here $T[1]^l M$ is understood as $Hom(\star\times D^l, M)$ for $l=k, k-1, k-2$, and $\sigma_i$ corresponds to degeneracies of $D^{l}$ restricted to $\star \times D^{l-1}$ for $l=k, k-1$. Let us explain a bit: Indicated by diagram \eqref{diag:hk-combi}, we will later consider $Hom(\star\times D^l, M)$ with $\star\times D^l\subset D^{l+1}$, instead of directly $\Hom(D^l,M)$. The inclusion $D^l\cong \star \times D^l\subset  D^{l+1}$ corresponds to the map 
\begin{equation*}
    \mathbb R[\epsilon_0,...,\epsilon_{l}]=C(D^{l+1})\to C(\star \times D^l)=\mathbb R[\epsilon_1,...,\epsilon_{l}], \quad \epsilon_0 \mapsto 0, \epsilon_j \mapsto \epsilon_j, \text{for} \; j=1, \dots, l. 
\end{equation*}
Here $\epsilon_i$'s all have degree -1, thus they anti-commute. The map $\sigma_i$  corresponds to degeneracies of $D^{l}$ restricted to $\star \times D^{l-1}$ for $l=k, k-1$. As a consequence, we have one more $\sigma_i: T[1]^kM /T[k]M\to T[1]^{k-1}M$ than we would have simply took $T[1]^k M= \Hom(D^k, M)$ and $T[1]^{k-1} M =\Hom(D^{k-1}, M)$ and understood $\sigma_i$ comes from $(s^D_i)^*$ on the level $k-2$ in \eqref{eq:pair-D}. In fact,  the maps $\sigma_1,...,\sigma_{k-1}$ correspond to the usual degeneracies $(s^D_0)^*,...,(s^D_{k-2})^*$ if we take $T[1]^k M= \Hom(D^k, M) $ and $T[1]^{k-1} M =\Hom(D^{k-1}, M) $, and the map $\sigma_0$, corresponds to the inclusion $D^{k-1}\to \star \times D^{k-1}\to D^k$.\\

Now we prove the key Lemma of this subsection. It is also the key step in repairing the mistake in Li's Lemma 8.23.  The proof is a long and fine combinatorial demonstration. Thus, we break it down into three observations to make the complex proof easier to understand and follow. 

\begin{lemma}\label{lem:key}
The N-manifold $T[1]^kM/T[k]M$ is the limit of the above family of diagrams  \eqref{diag:A} in the category  $\NMfd$. Moreover, the natural map $(\sigma_0, \dots, \sigma_{k-1}): T[1]^kM/T[k]M \to (T[1]^{k-1}M)^{[k-1]}$ is a strict embedding with respect to splittings determined by a chosen connection on $TM$.
\end{lemma}
\begin{proof}
Upon fixing a connection on $TM$, we can view these diagrams as diagrams in the category of $\Z^{<0}$-graded vector bundles and show that there is a limit there. We have a similar splitting of $T[1]^l M$ by $T^l$ as in \eqref{eq:T-k}, but as here we are taking $\Hom(\star \times D^l, M)$ instead of $\Hom( D^l, M)$ directly as therein, the indices are shifted by 1. That is, here $T[1]^l M$ has a splitting by $\oplus_{i=1}^{l}\tilde{T}^l_{-i}$ with
\begin{equation}\label{eq:tilde-T-l}
  \tilde{T}^l_{-i} = \bigoplus_{I\subset \{1, \dots, l\}: |I|=i} TM_I.  
\end{equation}

Since all morphisms involved are ($\Z^{<0}$-graded) vector bundle morphisms over the identity of $M$, the lemma boils down to show the following observation:
\begin{observation}\label{ob:explicit}
For any family of elements $\{v_0,...,v_{k-1}\}$ in $ \tilde{T}^{k-1}_\bullet$
---the graded vector bundle corresponding to $T[1]^{k-1}M=\Hom(\star \times D^{k-1}, M)$, satisfying compatibility equations 
\begin{equation}\label{eq:v-cond}
    \sigma_iv_{j+1}=\sigma_j v_i , \qquad \forall 0\leq i\leq j \leq k-2,
\end{equation}
there is a unique element $w\in T[1]M^{\binom{k}{1}} \oplus T[2]M^{\binom{k}{2}} \oplus \dots \oplus T[k]M^{\binom{k}{k-1}} \subset \tilde{T}^{k}_\bullet$  (which is the graded vector bundle corresponding to $T[1]^kM/T[k]M$)  mapping to them. 
\end{observation}

An element $v_i$ in $\tilde{T}^{k-1}_\bullet$ is given by its components $\{v_i^I\in TM_I\}$, where $I$ runs through all subsets of the set $\{1,...,k-1\}$.  The desired element $w$ will have components $\{w^I\in TM_I\}$ for all non-full subsets $I\subset \{1,...,k\}$. 

Let us try to understand the equation $\sigma_iw=v_i$ for a fixed multiindex $I\subset \{1,...,k-1\}$. We use notation in Notation \ref{not:index}. Notice that even though we have a shift, the notations therein stay with the same meaning. 

There are two options:
\begin{itemize}
    \item $i\not\in I$: Then the only contribution to $(\sigma_iw)^I$ comes from $w^J$ for $J=\overset{i\to}{I}$. 
    \item $i\in I$. In this case $I$ can still come from $\overset{i\to}{I}$, but also from $\overset{i-1\to}{I}$.
 \end{itemize}

Hence $\sigma_i w = v_i$ is equivalent to,  
\begin{align}\label{coordeq}
    v_i^I=\left\{ \begin{matrix}
    w^{\overset{i\to}{I}}&&i\not\in I\\
    w^{\overset{i\to}{I}}+w^{\overset{i-1\to}{I}}&&i\in I
    \end{matrix}
    \right.
\end{align}for all multi-indices $I\subset \{1, 2, \dots, k-1\}$ and all $i\in \{0,...,k-1\}$. Thus, our proof boils down to showing that linear equations in Eq.~\eqref{coordeq} provide a unique solution $w$ with fixed $v_i$ satisfying Eq.~\eqref{eq:v-cond} for $i\in \{0,...,k-1\}$. We prove this statement in two steps, which are summarized in Observation \ref{ob:eqs-wJ} and Observation \ref{ob:key}.

We first notice that
\begin{itemize}
    \item Components with different degree, namely index-length, in $\tilde{T}^{l}_{\bullet}$ do not interfere with each other. Especially, the component of $w$ of index-length $j$ only depends on the component of $v_i$'s with index-length $j$. 
\item There is no equation containing the highest degree $w^J$ with $J=(1,2,3,4,...,k)$. This is the reason why the pullback will be $T[1]^kM/T[k]M$ and not $T[1]^kM$.
\end{itemize}

Now we try to revert \eqref{coordeq} to get all equations to a form in which a fixed $w^J$ (i.e., a fixed multi-index $J$) is involved on the right-hand side. We do this by case-scanning.  Let us fix $J$ and $i$, we have the following case distinction:
\begin{enumerate}
    \item $i\in J$ and $i+1\in J$: Then $J$ can not be achieved as some $\overset{i\to}I$ or $\overset{i-1\to}I$. I.e., there is no equation involving $v_i$ and $w^J$.
    \item $i\in J$ and $i+1\not\in J$: Then $J$ can not be of the form $\overset{i-1\to}I$, but it can be of the form $J=\overset{i\to}I$, and in this case $I=\overset{i\leftarrow }{J}$. 
    This $v_i^{\overset{i\leftarrow }{J}}$ appears in the equation $v_i^{\overset{i\leftarrow }{J}}=w^{{J}}+ w^{\overset{i-1\to}{\overset{i\leftarrow }{J}}}$. The second summand (pulling everything after $i$ first, then pushing everything after $i-1$) turns out to change $i$ into $i+1$ and leave the rest as it is. We denote such operation by $\overset{i\curvearrowright i+1}{J}$. That is, we have $\overset{i-1\to}{\overset{i\leftarrow }{J}}=\overset{i\curvearrowright i+1}{J}$.  Therefore the unique equation involving $v_i$ and $w^J$ is:
    $$v_i^{\overset{i\leftarrow}{J}}=w^J+w^{\overset{i\curvearrowright i+1}{J}}.$$
   
    \item $i\not \in J$ and $i+1\in J$: This case is the opposite of the previous one, $J$ can not be $\overset{i\to}{I}$, but it can be realized as $\overset{i-1\to}I$ with $I=\overset{i-1\leftarrow }{J}$. We get the equation:
    $$v_i^{\overset{i-1\leftarrow}{J}}=w^{\overset{i\curvearrowleft i+1}{J}}+w^J$$
    \item $i\not \not\in J$ and $i+1 \not\in J$: This is the case, where \eqref{coordeq} has only one summand on the right-hand side, and we get $$v_i^{\overset{i\leftarrow}{J}}=w^J$$
\end{enumerate}

In summary for a multi-index $J$ ranging from $1$ to $k$ and an index $0\leq i\leq k-1$, we have the following observation:
\begin{observation}\label{ob:eqs-wJ}
A rewriting of Eq.~\eqref{coordeq} is given in the following cases:
\begin{center} 
\begin{tabular}{ c |c |c } 
  & $i\in J$ & $i\not\in J$\\ \hline
  $i+1\in J$& no equation & $v_i^{\overset{i-1\leftarrow}{J}}=w^{\overset{i\curvearrowleft i+1}{J}}+w^J$, if $i\ge 1$ \\   \hline
 $i+1\not\in J$ & $v_i^{\overset{i\leftarrow}{J}}=w^J+w^{\overset{i\curvearrowright i+1}{J}}$ & $v_i^{\overset{i\leftarrow}{J}}=w^J$.     
\end{tabular}
\end{center}
\end{observation}
Notice that the equation including $\overset{0 \curvearrowleft 1}{J}$ vanishes by definition. That is the reason why in the right top field, there is the additional assumption $i\ge 1$.\\ 

We are now going to show 
\begin{observation}\label{ob:key}
When $k\ge 2$, for any multi-index $J$ ranging from 1 to $k$ with length $l:=|J|<k$, the formulas in \eqref{coordeq} uniquely determine $\{w^J |~|J|=l\}$.
\end{observation}
We pick the lexicographical order on the multi-indices of length $l$ and do an induction-proof backwards.  

\begin{enumerate}
\item[A)] Initial case: For this, we only need to show that for the last multi-index $J=(k-l+1,...,k)$, we have
    \begin{enumerate}
        \item {\bf We can define $w^J=v_0^{\overset{0\leftarrow}{J}}$}, as $0, 1\not\in J$. 
        \item {\bf All  equations containing only $w^J$ in Observation \ref{ob:eqs-wJ} are satisfied if we define $w^J$ by (a).}
    \end{enumerate}   
\begin{proof}[Proof for A)] The formulas containing only $w^J$ are those of the type $w^J=v_i^{\overset{i\leftarrow}{J}}$, i.e., those $i>0$ where $i,i+1\not\in J$. For those, we have to show that
$$
v_i^{\overset{i\leftarrow}{J}}=v_0^{\overset{0\leftarrow}{J}}, 
$$
so that there is no contradiction among equations involving $w^J$ if $w^J$ is defined by (a). 

We consider the ${\overset{0\leftarrow}{I}}$ component of equation  $\sigma_0v_{i}=\sigma_{i-1}v_0$, where $I={\overset{i\leftarrow}{J}}={\overset{0\leftarrow}{J}}=(k-l,...,k-1)$. Since $0,i-1\not\in {\overset{0\leftarrow}{I}}$ and ${\overset{0\leftarrow}{I}}={\overset{i-1\leftarrow}{I}}$ we get $$
v_i^{\overset{i\leftarrow}{J}}=v_i^I=(\sigma_0v_{i})^{{\overset{0\leftarrow}{I}}}=(\sigma_{i-1}v_0)^{{\overset{i-1\leftarrow}{I}}}=v_0^I=v_0^{\overset{0\leftarrow}{J}}.
$$
\end{proof}

\item[B)] Induction Step: Fix a multi-index $J$ and assume that $w^{\tilde J}$ are defined and uniquely defined for all $\tilde J>J$ with respect to the lexicographical order,  and that all equations involving only $\{w^{\tilde J}| \tilde J>J\}$ in Observation \ref{ob:eqs-wJ} are satisfied. Then
    \begin{enumerate}
        \item {\bf We can define $w^J$ as follows:}
        Let $j=\max \{l\in\{1,...,k\} | l\not\in J\}-1$, then $j+1 \not\in J$. We set $w^J$ by:
        $$w^J=\left\{ \begin{array}{lr} v_j^{\overset{j\leftarrow}{J}} & \text{if}\; j\not\in J\\ v_j^{\overset{j\leftarrow}{J}}- w^{\overset{j\curvearrowright j+1}{J}} & \text{if}\; j\in J \end{array}\right.$$
        \item {\bf All equations involving only $\{w^{\tilde J}| \tilde J\geq J\}$ in Observation \ref{ob:eqs-wJ} are satisfied. }
    \end{enumerate}    
\begin{proof}[Proof for B)]
        In order to avoid the case distinction in the calculation, we  define 
        \begin{equation}\label{eq:shorten}
            w^{\overset{l\curvearrowright l+1}{J}}:=0,  \quad \text{ when } l\not\in J \text{ or }l+1\in J.
        \end{equation} 
        With this notation, the right bottom equation is a special case of the left bottom equation, thus we save the discussion of one case.
        The formulas involving $w^J$ and not involving any multi-index of $w$ which is smaller than $J$ are all of the type $v_i^{\overset{i\leftarrow}{J}}- w^{\overset{i\curvearrowright i+1}{J}} $, i.e., formulas with $v_i$ where $i+1\not \in J$, keeping in mind the notation \eqref{eq:shorten}.
        By our construction of $j$, it follows that $i<j$.  We  need to verify  
\begin{align}\label{eq:inductionstep}
    v_i^{\overset{i\leftarrow}{J}}-w^{\overset{i\curvearrowright i+1}{J}}=
v_j^{\overset{j\leftarrow}{J}}-w^{\overset{j\curvearrowright j+1}{J}},
\end{align}
so that $w^J$ is well defined by (a).

We first notice that we can rewrite the equation,  
\begin{align*}
    (\sigma_lv_k)^I=\left\{ \begin{matrix}
    v_k^{\overset{l\to}{I}}&&l\not\in I\\
    v_k^{\overset{l\to}{I}}+v_k^{\overset{l-1\to}{I}}&&l\in I
    \end{matrix}
    \right.
\end{align*}
to  $(\sigma_lv_k)^I=v_k^{\overset{l\to}{I}}+v_k^{\overset{l\curvearrowright l+1}{\overset{l\to}{I}}}$, using the notation with the curved arrows with convention \eqref{eq:shorten}.

Then for $i<j$,  we have $\sigma_iv_j=\sigma_{j-1}v_i$. On the component $I=\overset{j-1\leftarrow}{\overset{i\leftarrow}{J}}=\overset{i\leftarrow}{\overset{j\leftarrow}{J}}$, $\sigma_iv_j=\sigma_{j-1}v_i$ reads:  
\begin{equation} \label{eq:sig}
    \begin{split}
        &(\sigma_iv_j)^I=v_j^{\overset{i\to}{I}}+v_j^{\overset{i\curvearrowright i+1}{\overset{i\to}{I}}}=v_j^{\overset{j\leftarrow}{J}}+v_j^{\overset{i\curvearrowright i+1}{\overset{j\leftarrow}{J}}} \\
    =&(\sigma_{j-1}v_i)^I=v_i^{\overset{j-1\to}{I}}+v_i^{\overset{j-1\curvearrowright j}{\overset{j-1\to}{I}}}=v_i^{\overset{i\leftarrow}{J}}+v_i^{\overset{j-1\curvearrowright j}{\overset{i\leftarrow}{J}}}. 
    \end{split}
\end{equation} Here $\overset{i\to}{I}=\overset{j\leftarrow}{J}$ by definition of $I$. 

We will now prove  Equation \eqref{eq:inductionstep}, the proof works differently for $i+1=j$ and $i+1<j$:

When $i+1<j$, by Eq.~\eqref{eq:sig}, Eq.~\eqref{eq:inductionstep} reduces to
\begin{align*}
        v_i^{\overset{i\leftarrow}{\overset{j\curvearrowright j+1}{J}}}-w^{\overset{j\curvearrowright j+1}{J}} =
v_j^{\overset{j\leftarrow}{\overset{i\curvearrowright i+1}{J}}}-w^{\overset{i\curvearrowright i+1}J}
\end{align*}
By induction hypothesis, $w^{\overset{i\curvearrowright i+1}J} = v_j^{\overset{j\leftarrow}{\overset{i\curvearrowright i+1}J}}- w^{\overset{j\curvearrowright j+1}{\overset{i\curvearrowright i+1}J}}$ and $w^{\overset{j\curvearrowright j+1}J} = v_i^{\overset{i\leftarrow}{\overset{j\curvearrowright j+1}J}}- w^{\overset{i\curvearrowright i+1}{\overset{j\curvearrowright j+1}J}}$. As $i+1<j$, $w^{\overset{j\curvearrowright j+1}{\overset{i\curvearrowright i+1}J}}= w^{\overset{i\curvearrowright i+1}{\overset{j\curvearrowright j+1}J}}$. This concludes the proof for the case $i+1<j$.\\

When $i+1=j$, then we know $j=i+1, j+1=i+2\not\in J$. Equation \eqref{eq:inductionstep} then reads:
 $$   v_i^{\overset{i\leftarrow}{J}}-w^{\overset{i\curvearrowright i+1}{J}}=
v_j^{\overset{j\leftarrow}{J}}
$$
By induction hypothesis, we can now replace $w^{\overset{i\curvearrowright i+1}{J}}$ by $
v_j^{\overset{j\leftarrow}{\overset{i\curvearrowright i+1}J}} - w^{\overset{j\curvearrowright j+1}{\overset{i\curvearrowright i+1}J}}
$ and subsequently $w^{\overset{j\curvearrowright j+1}{\overset{i\curvearrowright i+1}J}}$ by $v_i^{ \overset{i\leftarrow}{\overset{j\curvearrowright j+1}{\overset{i\curvearrowright i+1}J}}}$ making the following equation the aim of proof:
$$
 v_i^{\overset{i\leftarrow}{J}}-v_j^{\overset{j\leftarrow}{\overset{i\curvearrowright i+1}J}} + v_i^{ \overset{i\leftarrow}{\overset{j\curvearrowright j+1}{\overset{i\curvearrowright i+1}J}}}=
v_j^{\overset{j\leftarrow}{J}}.
$$
Upon observing that $\overset{j\leftarrow}{\overset{i\curvearrowright i+1}J}=\overset{i\curvearrowright i+1}{\overset{j\leftarrow}J}$ and ${ \overset{i\leftarrow}{\overset{j\curvearrowright j+1}{\overset{i\curvearrowright i+1}J}}}=\overset{j-1\curvearrowright j}{\overset{i\leftarrow}{J}}$, this equation reduces exactly to Eq.~\eqref{eq:sig}. Thus, Eq.~\eqref{eq:inductionstep} is satisfied also in the $j=i+1$ case, completing the proof of the induction step, hence the proof of Observation \ref{ob:key}. 
\end{proof}
\end{enumerate}
Observation \ref{ob:key} implies Observation \ref{ob:explicit} and thus concludes the proof of this Lemma. 
\end{proof}

\begin{remark}\label{rm:MM}
We believe that \cite[Theorem 2.7]{madeleine-malte} is proving some similar result in another language and may be used to prove this Lemma if sufficient translation is provided. However, straightforwardly, the space $P$ therein is completely different from our limit $T[1]^kM/T[k]M$ as the maps $p_j^{\underline{n}
\backslash \{i\}}$ to form the limit are different from ours, namely the $\sigma_i$'s. Translating to our simplicial language, $p_j^{\underline{n} \backslash \{i\}}$ are degeneracy maps $D^{n-1}\to D^{n}$ of the nerve of the group $D\Rightarrow \star$ but not those of the nerve of the pair groupoid $D\times D \Rightarrow D$ as in our case. We leave the usage of this Theorem in the proof of our result to a future, possibly joint work.  
\end{remark}

\subsection{A monomorphism of $T[1]^k \Horn{k}{0}X$ to $(T[1]^{k}X_{k-1})^{[k]\backslash 0}$}\label{sec:B}
The map $d_J^X:X_k\to X_{k-1}^{[k]\backslash 0}$, for $J=(1, \dots, k)$, is easier to take care of, and one does not need a choice of connection.
Given a Lie $\infty$-groupoid $X_\bullet$, the maps $(d_1^X,....,d_k^X)$ factor through the horn $\Horn{k}{0}X$, which  can be seen as a subset of $X_{k-1}^{[k]\backslash 0}$. Moreover, it is exactly the limit of the following diagrams for $1\leq i<j\leq k$:

$$
\xymatrix{
\Horn{k}{0}(X) \ar[d]^{d_{j}^X} \ar[r]^{d_i^X} &X_{k-1}\ar[d]^{d_{j-1}^X}\\
X_{k-1}\ar[r]^{d_{i}^X} & X_{k-2}
}
$$
\begin{lemma}\label{lem:B}Let $X_\bullet$ be a Lie $\infty$-groupoid.  Then $ T[1]^k\Horn{k}{0}(X)$ is the limit of the following diagrams in the category of $\Z$-graded manifolds:
$$
\xymatrix{
T[1]^k\Horn{k}{0}(X)\ar[d]^{T[1]^k{d^X_{j}}} \ar[r]^{T[1]^k{d^X_i}} &T[1]^kX_{k-1}\ar[d]^{T[1]^k{d^X_{j-1}}}\\
T[1]^kX_{k-1}\ar[r]^{T[1]^k{d^X_{i}}} & T[1]^kX_{k-2}.
}
$$  
\end{lemma}
\begin{proof}
Here the proof follows directly from the fact that the functor $T[1]$ preserves limits, since it is the right-adjoint to the functor $ D\times $, i.e., $\hom(D\times A, B)=\hom(A,T[1]B)$.
\end{proof}
\begin{remark}\label{rm:mono}
This Lemma gives us a map $T[1]^k\Horn{k}{0}(X) \xrightarrow{T[1]^k i_{{k},{0}}} (T[1]^{k}X_{k-1})^{[k]\backslash 0}$ of $\Z$-graded manifolds over the embedding $\Horn{k}{0}(X) \xrightarrow{i_{{k},{0}}} X_{k-1}^{[k]\backslash 0}$. Since $T[1]$ is the right-adjoint to the functor $D\times$ and right-ajoints preserves monomorphism, $T[1]^k i_{{k},{0}}$ is a monomorphism in $\Z\Mfd$. Recall that monomorphism means left-cancellation (see Remark \ref{rm:strict}). To have left-cancellation in $\Z\Mfd$ implies having left-cancellation in its subcategory $\NMfd$. Thus, $T[1]^k i_{{k},{0}}$ is also a monomorphism in $\NMfd$.
\end{remark}

\subsection{The induction step} \label{sec:A-B}

We are now prepared to characterize the image of $T[1]^k X_k$ inside $(T[1]^{k} X_{k-1})^{[k]\backslash 0}  \times (T[1]^{k-1}X_k)^{[k-1]}$. The tangent maps of horn projection fit into the following commutative diagrams of N-manifolds
\begin{equation*}
    \xymatrix{T[k]X_k \ar[r] \ar[d]^{T[k]p^k_0} & T[1]^k X_k \ar[d]^{T[1]^k p^k_0} \\
    T[k] \Horn{k}{0}(X) \ar[r] & T[1]^k \Horn{k}{0}(X)}
\end{equation*} Since $p^k_0$ is a surjective submersion, according to Cor.\ref{cor:comp-conn}, we may choose compatible connections on $T^*X_k$ and $T^*\Horn{k}{0}(X)$ such that $T[1]^k p^k_0$ is a strict morphism, namely it comes from a morphism of graded vector bundles 
\begin{equation} \label{eq:slit-p-k}
  \bigoplus_{i=1}^k \bigoplus_{I\subset \{1, \dots, k\}: |I|=i} T(X_k)_I \xrightarrow{\oplus_i \oplus_I (Tp^k_0)_I}  \bigoplus_{i=1}^k \bigoplus_{I\subset \{1, \dots, k\}: |I|=i} T\Horn{k}{0}(X)_I.
\end{equation}
It is clear that the above map $\oplus_i \oplus_I (Tp^k_0)_I$ is a surjective submersion. Notice that $T[k]X_k$  
and $T[k] \Horn{k}{0}(X) $ are components when $i=k$ of the above direct sums   on the left and right respectively, 
and the map $T[k]p^k_0$ is a component of $\oplus (Tp^k_0)_I$. After mod out these components,  
\begin{equation} \label{eq:slit-p}
  \bigoplus_{i=1}^{k-1} \bigoplus_{I\subset \{1, \dots, k\}: |I|=i} T(X_k)_I \xrightarrow{\oplus_i \oplus_I (Tp^k_0)_I}  \bigoplus_{i=1}^{k-1} \bigoplus_{I\subset \{1, \dots, k\}: |I|=i} T\Horn{k}{0}(X)_I,
\end{equation} is again a surjective submersion. This shows that 
\begin{equation*}
    T[1]^k X_k /T[k]X_k \xrightarrow{(p^k_0)_*} T[1]^k\Horn{k}{0}(X)/ T[k] \Horn{k}{0}(X), 
\end{equation*}
is again a strict morphism and it comes from a surjective submersion of graded vector bundles thus a strict surjective submersion. It is also clear that the map $T[1]^k\Horn{k}{0}(X) \xrightarrow{q} T[1]^k \Horn{k}{0}(X)/T[k] \Horn{k}{0}(X)$ is also a strict surjective submersion.  Let 
\begin{equation*}
    A:=T[1]^k X_k / T[k] X_k,  \quad B:= T[1]^k \Horn{k}{0}(X), \quad C:=T[1]^k \Horn{k}{0}(X)/T[k] \Horn{k}{0}(X). 
\end{equation*}
Thus, with the above maps $p:=(p^k_0)_*$ and $q$, the limit $A\times_{ C} B$ is again an N-manifold which comes with a splitting given by a pair of compatible connections on $T^*X_k$ and $T^*\Horn{k}{0}(X)$ according to Section\ref{sec:limit}. We further show that:

\begin{lemma}\label{lem:ss}
The limit $A\times_C B \cong T[1]^kX_k/\ker Tp^k_0 [k]$. Thus, the natural map $T[1]^k X_k \to A\times_C B$ is a strict surjective submersion with respect to splittings determined by compatible connections on $T^*X_k$ and $T^*\Horn{k}{0}(X)$. 
\end{lemma}
\begin{proof}
We apply compatible connections to descend the argument to the level of graded vector bundles. Once it is descended to the level of graded vector bundles, things are very clear: $C$ contains all the information of $B$ except for the highest degree $T[k] \Horn{k}{0}(X)$. Notice that $T[k]X_k/\ker Tp^k_0 [k] = (p^k_0)^*T[k] \Horn{k}{0}(X) $,  a linear algebra calculation in the form of $ (V'/W')\times_{V/(W'/W)} V =V'/W$ gives us the desired result.
\end{proof}

Using the results in Section \ref{sec:key-A} and Section \ref{sec:B}, we want to obtain a map $i_A \times i_B: A\times_C B \to  (T[1]^{k-1} X_k)^{[k-1]} \times (T[1]^kX_{k-1})^{[k]\backslash 0}$ such that $H^{k-1} \xrightarrow{()\circ \huaF_{k-1}} (T[1]^{k-1} X_k)^{[k-1]} \times (T[1]^kX_{k-1})^{[k]\backslash 0} $ lifts to $A\times_C B$. For this, we need to understand the nature of $A\times_C B$ as the limit of a diagram system.

\begin{lemma} \label{lem:abc}
The fiber product $A\times_C B$ is the limit of the following diagrams in the category of N-manifolds:
$$
\xymatrix{
A\times_C B \ar[d]^{(s^D_{i'+1})^*} \ar[r]^{(s^D_i)^*} & T[1]^{k-1}X_k\ar[d]^{(s^D_{i'})^*}\\
T[1]^{k-1}X_k\ar[r]^{(s^D_{i})^*} & T[1]^{k-2}X_k
}, 
\xymatrix{
A\times_C B \ar[d]^{(d^X_{j})_*} \ar[r]^{(d^X_{j'})_*} &T[1]^kX_{k-1}\ar[d]^{(d^X_{j-1})_*}\\
T[1]^kX_{k-1}\ar[r]^{(d^X_{j'})_*} & T[1]^kX_{k-2}
}$$
$$
\xymatrix{
A\times_C B \ar[d]^{(s^D_j)^*} \ar[r]^{(d^X_i)_*} &T[1]^kX_{k-1}\ar[d]^{(s^D_j)^*}\\
T[1]^{k-1}X_{k}\ar[r]^{(d^X_{i})_*} & T[1]^{k-1}X_{k-1}.
}
$$
More precisely, for a test object $T \in \NMfd$, a collection of maps $a_I=(a_0,...,a_{k-1}):T\to T[1]^{k-1}X_k$, $b_J=(b_1,...,b_k):T\to T[1]^kX_{k-1}$ satisfying the following equations lift to a unique map $T\to A\times_C B$: 
\begin{itemize}
    \item[(a)] $(s^D_{i'})^*a_i=(s^D_i)^*a_{i'+1}$, for $0\leq i\leq i'\leq k-2$,
    \item[(b)] $(d^X_{j'})_*b_j=(d^X_{j-1})_*b_{j'}$, for $1\leq j'<j\leq k$,
    \item[(c)] $(d^X_j)_*a_i=(s^D_i)^*b_j$, for $j\in J=\{1,...,k\}, i\in I=\{0,...,k-1\}$.
\end{itemize}
Conditions (a), (b) and (c) correspond to the first, second, and third types of diagrams respectively, and the range of the indices in these diagrams are indicated in the corresponding conditions. 
\end{lemma}
\begin{proof}
Given a test object $T$ with a collection of maps $a_I, b_J$ with $I=\{0, \dots, k-1\}$ and $J=\{1, \dots, k\}$,  condition (a) and (b) are the conditions for $a_I: T \to (T[1]^{k-1}X_k)^{[k]\backslash 0}$ and $b_J: T\to (T[1]^kX_{k-1})^{[k-1]} $ to admit lifts to $a: T\to A$ and $b: T\to B$ respectively. We observe the following diagram:
\begin{equation}
    \xymatrix{
A\times_C B \ar[rr]^{pr_A}  \ar[d]^{pr_B}
&& B \ar[rr]^{i_B} \ar[d]^q
&& (T[1]^kX_{k-1})^{[k]\backslash 0} \ar[d]\\
A \ar[rr]^p \ar@{^{(}->}[d]^{i_A}
&& C \ar[rr] \ar@{^{(}->}[d]^{i_C}
&& (T[1]^kX_{k-1}/T[k]X_{k-1})^{[k]\backslash 0} \ar[d]\\
(T[1]^{k-1}X_k)^{[k-1]} \ar[rr]
&& (T[1]^{k-1}\Horn{k}{0}(X))^{[k-1]} \ar@{^{(}->}[rr]^{\scaleobj{0.8}{(T[1]^{k-1}i_{k,0})^{[k-1]}}}
&& (T[1]^{k-1}X_{k-1})^{[k-1]\times [k]\backslash 0}\\
}
\end{equation}
Here, $i_A$ and $i_C$ are the strict embeddings described in Section\ref{sec:key-A}, taking $M$ being $X_k$ and $\Horn{k}{0}(X)$ respectively\footnote{Notice that $(s^D_i)^*$ is exactly $\sigma_i$ therein.}; $T[1]^{k-1}i_{k,0}$ is the monomorphism  described in Section\ref{sec:B}, the upper-left square is a pullback diagram\footnote{To simplify the notation, various $T[1]^k d_j^X$ or $T[1]^{k-1} d_j^X$ therein are simply denoted by $(d_j^X)_*$.}; and the remaining three squares are commutative diagrams. The composition of two bottom maps is $[k-1]$-tuples of $(d^X_J)_*$ and the composition of the two vertical maps on the most right is $[k]\backslash 0 $-tuples of $(s^D_I)_*$. Thus, we have 
$$ {(s^D_I)^*}^{[k]\backslash 0}\circ  i_B\circ  b = {(s^D_I)^*}^{[k]\backslash 0}\circ b_J = (d^X_J)_*^{[k-1]} \circ a_I = (d^X_J)_*^{[k-1]} \circ i_A \circ a,  $$ implied by condition (b), (c), (a) sequentially. Chasing through the above diagram, we have  
\begin{equation}\label{eq:chasing-result}
    (T[1]^{k-1} i_{k, 0})^{[k-1]} \circ i_C \circ q\circ b = (T[1]^{k-1} i_{k, 0})^{[k-1]} \circ i_C \circ  p\circ a .
\end{equation}To show that $A\times_C B$ is the limit, we only need to show that $q\circ b = p\circ a$. Then  the starting data of $a_I, b_J$ give us unique $a, b$, which further satisfy $q\circ b = p\circ a$,  thus we have a unique lift $f: T\to A\times_C B$. 

Since we have \eqref{eq:chasing-result}, to show that $q\circ b = p\circ a$, it is sufficient to see that  $(T[1]^{k-1} i_{k, 0})^{[k-1]} \circ i_C$ is a monomorphism. The strict embedding $i_C$ is certainly a monomorphism in $\NMfd$ and $T[1]^{k-1}i_{k,0}$ is also a monomorphism in $\NMfd$  by Section\ref{sec:B},  thus  $(T[1]^{k-1} i_{k, 0})^{[k-1]} \circ i_C$ is a monomorphism. 
\end{proof}

Now we show the induction step in the following lemma:
\begin{lemma}\label{lem:ind-step}
Assume that $H^{k-1}$ is representable as a graded manifold $\mathcal M$ over $X_0$. Then $H^k$ is representable as a graded manifold by $\mathcal M\times_{X_0} \ker T(p^k_{0})[k]$.
\end{lemma}
\begin{proof}
Our starting point is the pullback diagram \eqref{diag:hk}. The right vertical map factors through $A\times_C B$ by Lemma \ref{lem:abc} and Lemma \ref{lem:ss}, 
\begin{equation}\label{eq:lift}
\xymatrixcolsep{15pc}\xymatrix{H^k \ar[r]^{\huaF_k} \ar[dd] & T[1]^{k}X_k  \ar@/^2.0pc/[dd]^{((s_I^D)^*, (d_J^X)_*) } \ar@{->>}[d]  \\
& A\times_C B \ar[d]^\iota \\
    H^{k-1} \ar[r]^-{((s_I^X)_*, (d_J^D)^*) \circ \huaF_{k-1}} \ar@{.>}[ur]^{\exists!} & (T[1]^{k-1}X_k)^{[k-1]}   \times  (T[1]^{k} X_{k-1})^{[k]\backslash 0}  }.
\end{equation}

We first show that the bottom map $H^{k-1}\to (T[1]^{k} X_{k-1})^{[k]\backslash 0}  \times (T[1]^{k-1}X_k)^{[k-1]}$  lifts to a unique map $H^{k-1}\to T[1]^kX_k/\ker(Tp^k_{0})[k]$. For this, we only need to verify that $a_i=(s^X_i)_*\circ \mathcal F_{k-1}$ and $b_j=(d^D_j)^*\circ \mathcal F_{k-1}$ satisfy condition (a), (b), (c).

For condition (a),  we look at the following diagrams, where the commutativity of the left square follows from the diagram \eqref{diag:hk} one level below, and that of the right square follows from simplicial identities.

$$
\xymatrix{
H^{k-1} \ar[d]^\pi \ar[r]^{\mathcal F_{k-1}}&
T[1]^{k-1}X_{k-1} \ar[r]^{(s^X_i)_*} \ar[d]^{(s^D_{i'})^*}&
T[1]^{k-1}X_k \ar[d]^{(s^D_{i'})^*}\\
H^{k-2} 
 \ar[r]^{(s^X_{i'})_* \circ \huaF_{k-2}}&
T[1]^{k-2}X_{k-1}\ar[r]^{(s^X_i)_*}&
T[1]^{k-2}X_k
}
$$
This means that $(s^D_{i'})^*a_i=(s^D_{i'})^*(s^X_i)_*\mathcal F_{k-1}=(s^X_i)_*(s^X_{i'})_*\pi$.
In particular, for $i\leq {i'}$:$$(s^D_{i'})^*a_i=(s^X_i)_*(s^X_{i'})_*\pi=(s^X_{{i'}+1})_*(s^X_i)_*\pi=(s^D_{i})^*a_{{i'}+1}.$$

Conditions (b) and (c) follow with a similar argument by simplicial identities and induction hypothesis. 

Secondly, we look at the composed map 
\begin{equation}
    \iota: A\times_C B \xrightarrow{pr_A\times pr_B} A\times B \xrightarrow{ i_A \times i_B} (T[1]^{k-1} X_k)^{[k-1]} \times (T[1]^kX_{k-1})^{[k]\backslash 0}.
\end{equation} By Section\ref{sec:B}, $i_B:= T[k] i_{k, 0}$ is a monomorphism,  and by Section\ref{sec:key-A}, $i_A$ is a strict embedding thus a also monomorphism. Moreover,  it is a general fact on fiber product that $A\times_C B \xrightarrow{ pr_A \times pr_B} A\times B$ is a monomorphism. We shortly recall the proof here: Let $g_1, g_2: T\to A\times_C B$ such that $pr_A \times pr_B \circ g_1= pr_A \times pr_B \circ g_2$. Then, in particular, $pr_A \circ g_1 = pr_A \circ g_2$ and the same for $pr_B$. Thus, $q\circ pr_A \circ g_1 = q\circ pr_A \circ g_2= p\circ pr_B \circ g_1 = p\circ pr_B \circ g_2$, thus there exists a unique map $g_1=g_2: T\to A\times_C B$ by the definition of fiber product. Thus, $\iota$ is a monomorphism. Following a diagram chasing argument, one can see that $H^k$ is also the fiber product of the smaller angled diagram, that is, 
\begin{equation} \label{eq:hk-iso}
    H^k\cong H^{k-1}  \times_{A\times_C B} T[1]^k X_k . 
\end{equation}
Notice that 
\begin{equation}
    T[1]^kX_k \cong T[1]^k X_k /\ker Tp^k_0[k] \times_{X_k} \ker Tp^k_0[k] \cong A\times_C B \times_{X_k} \ker Tp^k_0[k].
\end{equation}
Moreover, fiber products in $\NMfd$ still satisfies base change via a typical diagram chasing argument. Thus, we have 
\begin{equation} \label{eq:hk-k-1}
\begin{split}
       H_k &\cong H_{k-1} \times_{A\times_C B} T[1]^k X_k \cong H_{k-1} \times_{A\times_C B} A\times_C B \times_{X_k} \ker Tp^k_0[k] \\ 
       & \cong H_{k-1} \times_{X_k}  \ker Tp^k_0[k] \cong H_{k-1} \times_{X_0} X_0 \times_{X_k}  \ker Tp^k_0[k] = H_{k-1} \times_{X_0} \ker Tp^k_0[k]|_{X_0} 
\end{split}
\end{equation}
\end{proof}

Now Lemma \ref{lem:induction} follows:
\begin{proof}[Proof for Lemma  \ref{lem:induction}]
The case $k=0, 1$ are verified in Lemma \ref{lem:case0} and Section \ref{sec:k=1}. Now by induction hypothesis, we have already $H_{k-1} \cong \mathfrak S (\oplus_{i=1}^{k-1} Tp^i_0[i]|_{X_0} ) $, it is clear that $H_{k}\cong \mathfrak S( \oplus_{i=1}^{k} Tp^i_0[i]|_{X_0} )$ by Lemma \ref{lem:ind-step}. Notice that in the calculation, for simplicity, we sometimes omit $\mathfrak S$, namely, a graded vector bundle also means its represented graded manifolds. 
\end{proof}
This completes the proof of Theorem \ref{thm:ker-p}. 

\begin{remark}
It is clear that for each step when $k\ge 2$, we need to choose a pair of compatible connections on $TX_k$ and $T\Horn{k}{0}(X)$ (see Lemma \ref{lem:ss}). But we do not need any connection for $k=0, 1$. Thus, the splitting of $\huaT (X_\bullet)$ depends on a choice of compatible connections on  $TX_k$ and $T\Horn{k}{0}(X)$, and isomorphisms \eqref{eq:hk-iso}  for $k=2, \dots, n$. 
\end{remark}

\section{Examples}
\begin{example}[Lie groupoids]
In the case when $X_\bullet$ is the nerve of a Lie groupoid,  the horn projection $p^1_0: X_1 \to X_0$ is simply $d_1=\bs$ the source of the Lie groupoid. Thus, the tangent complex of a Lie groupoid $X_1 \rightrightarrows X_0$ is simply $\ker T\bs|_{X_0}$, the vector bundle of the Lie algebroid. This especially tells us the tangent complex of a Lie group $G$ is the Lie algebra $\mathfrak{g}$. 
\end{example}

In a forthcoming paper, we will demonstrate how to obtain (higher) Lie brackets on the tangent complexes. Now we focus on examples of tangent objects of various  Lie $n$-groupoids.

\begin{example}[Courant Lie 2-groupoid]
The integration of Courant algebroids is a hot topic. Various integration models for special cases are given in \cite{libland-severa1, MeTa11, MeTa18a, shengzhu3}. As the nature of a Courant algebroid is a 2-shifted symplectic Lie 2-algebroid, one expects a 2-shifted symplectic Lie 2-groupoid as the integration object. As the topic of this article is differentiation rather than integration, we do not look for a 2-shifted symplectic Lie 2-groupoid structure. However, we take the finite-dimensional model of a Lie 2-groupoid integrating a standard Courant algebroid, which we call a {\bf Courant Lie 2-groupoid}, and calculate its tangent object. Notice that the finite-dimensional models in the above literature are locally the same. The one we use here is from \cite{shengzhu3}. Given a manifold $M$, the Courant Lie 2-groupoid is 
\begin{equation}\label{eq:integration-gpd}
\xymatrix{X_2:(\Pi_1(M)\times_M\Pi_1(M))\times_{M^{\times 3}} (T^*M)^{\times 3}
  \ar@<.7ex>[r] \ar@<-.7ex>[r] \ar[r] & X_1:= \Pi_1(M)\times_{M} T^*M
  \ar@<.5ex>[r] \ar@<-.5ex>[r] & X_0:= M,   }
\end{equation}

\begin{center}
\begin{tikzpicture}
\draw (-2.2,0.8) node [label=right:\rm{\{}] {};

\draw (0,0) node [draw,circle,fill=black,minimum size=2pt, inner
sep=0pt] (0) [label=right:$x_0$] {}
 ++(0:2cm) node [draw,circle,fill=black,minimum size=2pt, inner
sep=0pt] (2) [label=right:$x_2$] {}; \draw (1,1.8) node
[draw,circle,fill=black,minimum size=2pt, inner sep=0pt] (1)
[label=left:$x_1$] {};

\draw[<-](0) to [out=80,in=200] node{$\gamma_{0,1}$}
(1);\draw[<-](0) to [out=20,in=160] node{$\gamma_{0,2}$} (2);
\draw[<-](1) to [out=-30,in=100] node{$\gamma_{1,2}$} (2);

 \draw[->]
(0) to (-0.8,0.3) node
 [label=left:$\xi^{0,1}$] {};\draw[->]
(0) to (-0.8,-0.3) node
 [label=left:$\xi^{0,2}$] {};
 \draw[->]
(1) to (1,2.6) node
 [label=left:$\xi^{1,2}$] {};
 \draw (2.2,0.8) node [label=right:\rm{\}}] {};

 \draw[->] (2.8,1) to (3.6,1);
 \draw[->] (2.8,0.8) to (3.6,0.8);
 \draw[->] (2.8,0.6) to (3.6,0.6);

\draw (3.4,0.8) node [label=right:\rm{\{}] {};

 \draw(4.5,0.5) node
[draw,circle,fill=black,minimum size=2pt, inner sep=0pt] (0)
[label=left:$x_0$] {}
 ++(0:2cm) node [draw,circle,fill=black,minimum size=2pt, inner
sep=0pt] (1) [label=right:$x_1$] {}; \draw[<-](0) to [out=20,in=160]
node{$\gamma$} (1); \draw[->] (0) to (4.5,1.3) node
 [label=left:$\xi$] {};

 \draw(6.8,0.8) node [label=right:\rm{\}}] {};

\draw[->] (7.3,0.9) to (8.1,0.9);
 \draw[->] (7.3,0.7) to (8.1,0.7);

 \draw(8.5,0.8) node [draw,circle,fill=black,minimum size=2pt, inner
sep=0pt] (0) [label=right:\rm{\}}][label=left:\rm{\{}]{};
\end{tikzpicture}
\end{center}
where $\Pi_1(M)=\tilde{M}\times \tilde{M}/\pi_1(M)$ is the
fundamental groupoid of $M$ with $\tilde{M}$ the simply connected
cover  of $M$. 

We notice that $p^1_0: \Pi_1(M) \times_M T^*M \to M $ is given by $(\gamma, \xi) \mapsto \gamma(0)$. We have $\ker Tp^1_0|_M=T^*M \oplus TM$. Similarly, $p^2_0: X_2 \to \Horn{2}{0}(X)$ is given by
\[ (\gamma_{1,2}, \gamma_{0,1}, \xi^{1,2}, \xi^{0,1}, \xi^{0,2}) \mapsto (\gamma_{0, 1}, \gamma_{0,1}\cdot \gamma_{1,2}, \xi^{0,1}, \xi^{0,2}). 
\] Thus, $\ker Tp^2_0|_M = T^*M$. Thus, we obtain $T[1]M\oplus T^*[1]M \oplus T^*[2]M$, which is the underlying graded vector bundle of a standard Courant algebroid, as the tangent of the Courant Lie 2-groupoid.

\end{example}

\begin{example}[Semi-strict Lie 2-groupoids]
The above example of a Courant Lie 2-groupoid is a special case of a semi-strict Lie 2-groupoid. A Lie 2-groupoid can be viewed as \cite{z:tgpd-2} a groupoid object in
 the 2-category \GpdBibd\; where the space of objects is only a manifold (but not a general Lie groupoid). Here \GpdBibd\; is the 2-category with Lie groupoids as objects,
  Hilsum-Skandalis (HS) bimodules as morphisms, and isomorphisms of HS bimodules as 2-morphisms.  A \emph{semistrict Lie 2-groupoid} is a particular Lie 2-groupoid, which is a groupoid object in the
  2-category of \Gpd \; with  the space of objects a manifold,  where
  \Gpd \; is a sub-2-category of \GpdBibd \; containing only strict groupoid morphisms as morphisms. 
In particular, this means that a semistrict Lie 2-groupoid consists of a Lie groupoid $\huaG:=G_1 \underset{\bt_G}{\overset{\bs_G}{\Rightarrow}} G_0$ over a manifold $M$ with $\bs, \bt: G_0 \Rightarrow M$, such that there is an additional horizontal multiplication morphism $m: \huaG \times_{M} \huaG \to \huaG $ of Lie groupoids, which is associative up to a 2-morphism called an associator. For the complete definition, we refer to \cite[Definition 5.2]{shengzhu3}. 

The Lie 2-groupoid $X_\bullet$ described using the simplicial language corresponding to the semistrict Lie 2-groupoid $\huaG \Rightarrow M$ is given by 
\[X_0= M, \quad X_1 = G_0, \quad X_2 = \left(X_1 \times_{s,X_0,t} X_1\right) \times_{m_0, X_1, \bt_G} G_1, 
\]where $m_0$ is the part of $m$ on the level of objects. We refer to \cite[Eq.~(19) Section4.1]{z:tgpd-2} for face and degeneracy maps. Then $p^1_0=\bs: X_1 \to X_0$. Thus, $\ker Tp^1_0|_{X_0} = \ker T\bs|_M$. Moreover, 
\[
p^2_0: X_2= \left(X_1 \times_{\bs,X_0,\bt} X_1\right) \times_{m_0, X_1, \bt_G} G_1 \to X_1 \times_{X_0} X_1, \quad (\gamma_{12}, \gamma_{01}, \xi) \mapsto (\gamma_{01}, \bs_G(\xi) ).
\]
Here we use some notation inspired by the calculation in the last example. Since $\bt_G (\xi) = m_0(\gamma_{01}, \gamma_{12})$, when we compare the information on the left-hand side with that on the right-hand side, we obtain that $\ker Tp^2_0|_M = \ker T\bs_G|_M$. Here we use the embedding  $M \xrightarrow{s_0} G_0=X_1 \xrightarrow{s_0} X_2$. Thus, we obtain $$\ker T\bs[1]|_M \oplus \ker T\bs_G[2]|_M, $$ 
as the tangent object of a semistrict Lie 2-groupoid.
\end{example}

\begin{example}[Strict Lie 2-groups and crossed modules of Lie groups]
A crossed module $H \xrightarrow{\partial} G$ gives rise to a strict Lie 2-group whose underlying Lie groupoid is $G\times H \Rightarrow G$, which is a special example of semistrict Lie 2-groupoid with $M=pt$. Thus, the previous example especially implies that its tangent is 
\[\g[1]\oplus \mathfrak h[2].
\]
\end{example}

\begin{example}[general central extensions for Lie 2-groups]
Let $G$ be a Lie group, $\mu:A\to B$ a homomorphism of abelian Lie groups. We consider the simplicial manifold $\mathcal BG$, which is the nerve of the groupoid $G\rightrightarrows *$ .  Let $(U^{(n)}_i)_{i\in I^{(n)}}$ be a simplicial covering of $\mathcal BG$ in the sense of {\cite{WoZh16}} and $\phi\in {Z}^3_U(\mathcal BG,A\to B)$. We will follow the construction  of \cite{WoZh16} of a stacky Lie groupoid central extension corresponding to cocycle $\phi$ and then \cite{z:tgpd-2} to turn the stacky construction into a Lie 2-group $\huaG$ in the simplicial formalism as in Def. \ref{def:lie-n-gpd}.
\begin{equation} \label{eq:central-ext-G}
    1 \to (A\to B) \to \huaG \to G \to 1
\end{equation}

In fact, given a stacky 2-group $(\Gamma_1\underset{\bs}{\overset{\bt}{\rightrightarrows}} \Gamma_0)\rightrightarrows *$, the corresponding simplicial 2-group is given by $X_0=*$, $X_1=\Gamma_0$ and $X_2=E_m$, the bibundle representing the multiplication of $\Gamma$. In our case $\Gamma=\Gamma[\phi]$ is given by $\Gamma_0= U^{(1)}_{[0]}\times B$ and   $\Gamma_1= U^{(1)}_{[1]}\times B\times A$, where $U^{(1)}_{[k]}$ denotes all possible {$k$-fold intersections} of $U^{(1)}$  
$$\bigcap_{j=0}^{k-1} U^{(1)}_{i_j}=U^{(1)}_{i_0}\times_G ...\times_GU^{(1)}_{i_{k-1}}.$$
We have source and target maps $\bs(u,b,a)=(\pi_0u,b)$ and $\bt(u,b,a)=(\pi_1u,b)$ of $\Gamma$. The zigzag giving rise to the multiplication is $\Gamma\times \Gamma\leftarrow \Gamma^2\rightarrow \Gamma$, where $\Gamma^2$ is the groupoid
$$
\Gamma^2_1= U^{(2)}_{[1]}\times B^2\times A^2\rightrightarrows \Gamma^2_0=U^{(2)}_{[0]}\times B^2
$$
with analogous source and target maps. The bibundle $E_m$ then is given by 
$$((\Gamma_1\times \Gamma_1)\times_{t,\Gamma_0\times \Gamma_0}\Gamma^2_0\times_{\Gamma_0,t}\times \Gamma_1)/\Gamma^2_1.$$ We need to specify the groupoid morphisms from $\Gamma^2$. As a map from $\Gamma^2_1$ to $(\Gamma\times \Gamma)_1=\Gamma_1\times\Gamma_1$ we have:
$$(v_0,v_1,b_0,b_1,a_0,a_1)\mapsto (d_0v_0,d_0v_1,b_0,a_0)\times (d_2v_0,d_2v_1,b_1,a_1).$$ 
The map $\Gamma^2_1\to \Gamma_1$ is given by $$(v_0,v_1,b_0,b_1,a_0,a_1)\mapsto (d_1v_0,d_1v_1,b_0+b_1+ \phi^{2,0,0}(v_0), a_0+a_1+ \phi^{2,1,-1}(v_0,v_1)).$$
 On the base $\Gamma^2_0$ the maps are  $(v_0,b_0,b_1)\mapsto (d_0v_0,b_0)\times (d_2v_0,b_1)$ to and $(v_0,b_0,b_1)\mapsto (d_1v_0,b_0+b_1+ \phi^{2,0,0}(v_0))$respectively. Hence $$E_m=[((U_{[1]}^{(1)})^{3}\times_{(U_{[0]}^{(1)})^{3}} U^{(2)}_1)/U_{[1]}^{(2)}] \times B^2\times A.$$
Since the first factor is just a cover of a small region in $G\times G$, this means that $ker(Tp^2_0)|_{X_0}=T_eA$. Since $X_0=*$, we also have $ker(Tp^1_0)|_{X_0}=T_eG\oplus T_eB$. Thus, the tangent object for the central extended Lie 2-group $\huaG$ is
\[
\g[1]\oplus \mathfrak b [1] \oplus \mathfrak a[2],
\]where $\g$, $\mathfrak b$ and $\mathfrak a$ are the Lie algebra of $G$, $B $ and $A$ correspondingly. 
\end{example}

\begin{example}[String Lie 2-group]
A famous sub-example of the above is the case of String Lie 2-group $\String(G)$. Given $G=Spin(n)$, or in general (see \cite[Theorem 1.47, 1.81]{tanre}) a compact, connected and simple Lie group $G$, one can form its associated String Lie 2-group $\String(G)$, which is a central extension \[
1\to \huaB S^1 \to \String(G) \to G \to 1,
\] of Lie 2-groups in the sense of \cite{schommer:string-finite-dim}, with extension class the generator of $H^2(\huaB G, \huaB S^1)= H^3(\huaB G, S^1)=H^4(\huaB G, \Z)=\Z$. It fits into the last example with $A=S^1$ and $B=\{1 \}$. Thus, the tangent object for $\String(G)$ is
\[
\g[1] \oplus \R[2].
\]
\end{example}

\begin{example}[$n$-tower]
Given a complex of Lie groups
\begin{equation}
  G \xleftarrow{q_2} \Pi_2 \xleftarrow{q_3} \Pi_3 \xleftarrow{q_4} \dots \xleftarrow{q_n} \Pi_n,
\end{equation}
together with  $G$-actions $\alpha_i: G\to \Aut(\Pi_i)$ such that the $q_i$'s are $G$-equivariant with respect to the conjugation action of $G$ on $G$ itself, we now construct a Lie $n$-group inductively with the help of $n-1$ cocycles $c_i$'s which will be defined inductively too. We call these data an $n$-tower, which shows up in the construction of a sigma model for topological orders \cite{Lan-Zhu-Wen:2019}. 

The first cocycle $c_3\in Z^3(G, (\Pi_2^0)^{\alpha_2})$ is a smooth normalized 3-cocycle of $G$ with coefficients  in $\Pi_2^0:=\ker q_2$ with action $\alpha_2$, that is, $c_3: G^{\times 3} \to \Pi_2^0$ is a smooth map, and 
\[
\delta_{\alpha_2} c_3 (g_1, g_2, g_3, g_4) = \alpha_2(g_1) c_3(g_2, g_3, g_4) - c_3(g_1 g_2, g_3, g_4) + c_3(g_1, g_2 g_3, g_4) - c_3(g_1, g_2, g_3 g_4) + c_3(g_1, g_2, g_3) =0.
\]
With $c_3$, we can build up a Lie 2-group $\huaG_{c_3}(G, \Pi_2)$ as in \cite[Section8]{baez:2gp}. This Lie 2-group written in the simplicial Kan complex $X^{(2)}_\bullet$ has the form: $X^{(2)}_0=pt, X^{(2)}_1 = G$,  
\[
\begin{split}
    X^{(2)}_2 := \{ &(x^1_{01}, x^1_{02}, x^1_{12}; x^2_{012})| x^1_{..}\in G, x^2_{012}\in \Pi_2,  x^1_{01} x^1_{12} (x^1_{02})^{-1} = q_2(x^2_{012})\} \cong  G\times G \times \Pi_2, \\
    X^{(2)}_3:=\{ &(x^1_{01}, x^1_{02}, \dots, x^1_{23}; x^2_{012}, \dots, x^2_{123})| x^1_{..}\in G, x^2_{..}\in \Pi_2, \\
     & d x^1 = q_2(x^2), d_{\alpha_2} x^2 = c_3(x^1).  \} \cong  G^{\times (^3_1)} \times \Pi_2^{\times (^3_2)},
\end{split}
 \]
and higher levels $X^{(2)}_m \cong G^{\times m} \times \Pi_2^{\times (^m_2)}$ are determined by taking the coskeleton functor \cite[Section2.3]{z:tgpd-2}. Here $(dx^1)_{abc}:= x^1_{ab}x^1_{bc}(x^1_{ac})^{-1}$ and
\begin{equation}\label{eq:d-alpha-3-x}
    (d_{\alpha_2} x^2)_{abcd} := \alpha_2(x^1_{ab}) x^2_{bcd} - x^2_{acd}+x^2_{abd}-x^2_{bcd}
\end{equation}

Inductively, the second cocycle $c_4 \in Z^4(\huaG_{c_3}(G, \Pi_2), (\Pi_3^0)^{\alpha_3}) $ is a smooth normalized\footnote{in the sense that its evaluation on degenerate simplices is 0. } 4-cocycle of the Lie 2-group $\huaG_{c_3}(G, \Pi_2)$ with coefficients  in $\Pi_2^0:=\ker d_2$ with action $\alpha_3$, where the smooth Lie 2-group cochain complex is the one of its simplicial nerve. More precisely,   $c_4: X^{(2)}_4 \to \Pi_3^0$ is a smooth map, and 
\[
\delta_{\alpha_3} c_4 (x) = \alpha_3(g) c_4(d_0(x))+ \sum_{j=1}^4 (-1)^j c_4(d_j(x)) ,
\]where $g= d_2\circ d_3 \circ d_4 (x)$. With $c_4$, we build a Lie 3-group $\huaG_{c_4}(G, \Pi_2, \Pi_3)$ whose simplicial Kan complex $X^{(3)}_\bullet$ has the form: $X^{(3)}_0=X^{(2)}_0, X^{(3)}_1 = X^{(2)}_1$, $X^{(3)}_2 = X^{(2)}_2$, and
\[
\begin{split}
    X^{(3)}_3 := \{ &(x^1_{01}, x^1_{02}, \dots, x^1_{23}; x^2_{012}, \dots, x^2_{123}; x^3_{0123}): x^1_{..}\in G, x^j_{..}\in \Pi_j, j=2, 3, \\
     & d x^1 = q_2(x^2), d_{\alpha_2} x^2 = q_3(x^3) +c_3(x^1).  \} \cong  G^{\times (^3_1)} \times \Pi_2^{\times (^3_2)} \times \Pi_3^{\times (^3_3)}, \\
    X^{(3)}_4 := \{ &(x^1_{01}, x^1_{02}, \dots, x^1_{34}; x^2_{012}, \dots, x^2_{234}; x^3_{0123}, \dots, x^3_{1234}): x^1_{..}\in G, x^j_{..}\in \Pi_j, j=2, 3, \\
     & d x^1 = q_2(x^2), d_{\alpha_2} x^2 = q_3(x^3) +c_3(x^1), d_{\alpha_3}x^3= c_4(x^1;x^2).  \} \cong  G^{\times (^4_1)} \times \Pi_2^{\times (^4_2)} \times \Pi_3^{\times (^4_3)}, 
\end{split}
 \]
where $d_{\alpha_3} x^3$ is a similar alternating sum in $\Pi_3$ as \eqref{eq:d-alpha-3-x}. This construction continues and it gives us inductively a Lie 2-group $X^{(2)}_\bullet$, a Lie 3-group $X^{(3)}_\bullet$, $\dots$, and a Lie $n$-group $X^{(n)}_\bullet$ (see \cite[Appendix L]{Lan-Zhu-Wen:2019} and \cite{daniel-thesis} for details). At the end of the day, we have
\begin{equation}
    \begin{split}
        X^{(n)}_m= & G^{\times (^m_1)} \times \Pi_2^{(^m_2)} \times \dots \times \Pi_m^{(^m_m)}, \quad \text{when} \; m\le n, \\
        X^{(n)}_m= & G^{\times (^m_1)} \times \Pi_2^{(^m_2)} \times \dots \times \Pi_n^{(^m_n)}, \quad \text{when} \; m\ge n. 
    \end{split}
\end{equation}
For $m\le n$, after removing the only $m$-simplex in $X^{(n)}_m$, the horn space is
\begin{equation}
    \Horn{m}{0}(X^{(n)}_\bullet) = G^{\times (^m_1)} \times \Pi_2^{(^m_2)} \times \dots \times \Pi_{m-1}^{(^m_{m-1})}, 
\end{equation}
containing $m$ $(m-1)$-simplices and $p^m_0$ is simply the projection using the last equation in $X^{(n)}_m$. Thus, $\ker Tp^m_0 = \mathfrak a_m$, where $\mathfrak{a}_m$ is the Lie algebra of $\Pi_m$ when $m\ge 2$ and $\ker Tp^1_0 = \mathfrak g$. Therefore, the tangent of $X^{(n)}_\bullet$ is $\g[1] \oplus \mathfrak{a}_2[2] \oplus \dots \oplus \mathfrak a_m[m]$. 
\end{example}

\begin{example}[strict extensions for Lie $n$-groups]
As pointed out in \cite{severa:diff},  similar to \cite[Section8]{baez:2gp} where a 3-cocycle is used to build a 2-group, if $n\ge 3$ and we are given a Lie group $G$, an abelian Lie group $\Pi$, a smooth $n+1$-cocycle $c: G^{n+1}\to \Pi$, and an action $\alpha: G\to \Aut(\Pi)$, we can build a Lie $n$-group $X_\bullet$, which may be viewed as an extension $G\ltimes_{c} \Pi$ of $G$ by $\Pi$,  with  
\[
X_m = G^{\times m}, \quad \text{when}\; m\le n-1, \quad X_m  = G^{\times m} \times \Pi^{\times (^m_n)} \quad \text{when}\; m\ge n. 
\]In fact, this is a special example of the above one when $\Pi_2, \dots, \Pi_{n-1}$ are 0 and $\Pi_n=\Pi$. Thus, its tangent is $\g[1]\oplus \mathfrak a[n]$ where $\mathfrak a$ is the Lie algebra of $\Pi$.  
\end{example}

\begin{example}[simplicial Lie groups and their simplicial classifying complex \cite{Jurco}]
Apart from the $n$-tower, there is a much more classical generalization of crossed modules, namely simplicial groups. For the purpose of differentiation, we are interested in a simplicial Lie group  $G_\bullet$, which is a group object in the category of simplicial manifolds. That is, each level $G_i$ is a Lie group and all the face and degeneracy maps are Lie group morphisms. 

A simplicial group $G_\bullet$ always satisfies the Kan conditions. This is proven in the set-theoretic setting in some classical literature \cite[Theorem 3]{Moore56}, and we refer to \cite{nlabkan} for a more complete history of the literature. The proof is to construct degenerate horn fillings explicitly, that is, for a horn $\lambda \in \Horn{k}{i}(G_\bullet)$, one constructs a simplex $g_k \in G_k$ such that $p^k_i(g_k)=\lambda$. The construction basically uses only a combination of face and degeneracy maps and multiplication in $G_\bullet$, thus it survives differential geometry.  Namely, if $G_\bullet$ is further a simplicial Lie group, for a smooth curve $\lambda(t) \in \Horn{k}{i}(G_\bullet)$, we have a smooth curve $g_k(t) \in G_k$ lifting it. Thus, for any preimage $g'_k \in G_k$ of $\lambda$ under $p^k_i$, $g'_k g_k^{-1} g_k(t)$ is a smooth curve    lifting $\lambda(t)$. Thus, a simplicial Lie group $G_\bullet$ satisfies Kan conditions with surjective submersions as in Def. \ref{def:lie-n-gpd}, namely $G_\bullet$ is an $L_\infty$-groupoid. Notice that $G_\bullet$ is usually not an $L_\infty$-group unless $G_0=\{1\}$. 

For a simplicial Lie group $G_\bullet$, its associated {\em Moore complex} (or normalized chain complex) is a $\Z^{\ge 0}$-graded chain complex  $((NG)_\bullet, \partial_\bullet)$  of (possibly non-abelian) groups with
\[
(NG)_k= \cap_{i=1}^k \ker d^k_i, \quad \partial_k= d^k_0: (NG)_k \to (NG)_{k-1}. 
\]
Let $\mathfrak n_k$ be the Lie algebra of $(NG)_k$. One can calculate that $\ker Tp^k_0|_{G_0}=\underline{\mathfrak n}_k$, which is a vector bundle over $G_0$ with fiber  $\mathfrak n_k$. Thus, the tangent object for $G_\bullet$ is
\[
\underline{\mathfrak n_1} [1] \oplus \underline{\mathfrak n_2} [2] \oplus \dots \oplus  \underline{\mathfrak n_k} [k] \oplus \dots .
\]
This result is also calculated in \cite[Remark 4.11]{cech:2016}. 

Associated to a simplicial group $G_\bullet$, its simplicial classifying complex $\bar{W}G$ is given by 
\[
\bar{W}G_k= G_{k-1}\times \dots \times G_0.
\]
We should imagine an element of $\bar{W}G_k$ as a sequence of arrows --- which are elements in $G_i$ between $x_j$'s, which are in turn all $1\in G_0$: 
\[
 \begin{tikzcd}
x_0  & x_1 \arrow[l, "g_{k-1}"', bend right] & x_2 \arrow[l, "g_{k-2}"', bend right] & \dots & x_{k-2} & x_{k-1} \arrow[l, "g_1"', bend right] & x_k \arrow[l, "g_0"', bend right]. 
\end{tikzcd}
\]We denote the face maps of $\bar{W}G_\bullet$ by $\bar{d}^k_i$, and they correspond to deleting $x_i$ in the above picture, that is
\[
    \bar{d}^k_i(g_{k-1}, g_{k-{2}}, \dots, g_1, g_0) = \begin{cases}
    (g_{k-2}, \dots, g_0), \quad i=0 \\
    (d_0(g_{k-1}) g_{k-2}, g_{k-3}, \dots, g_0), \quad i=1 \\
    \dots \\
    (d_{i-1}(g_{k-1}) , \dots, d_1(g_{k-i+1}), d_0(g_{k-i}) g_{k-i-1}, g_{k-i-2}, \dots, g_0) \\
    \dots & \\
    (d_{k-1}(g_{k-1}) , \dots, d_1(g_{1})), \quad i=k. 
\end{cases}
\]
Thus, we have
\begin{equation}\label{eq:ker-d-i}
\begin{split}
    \ker T\bar{d}^k_1 = & \{ (\delta g_{k-1}, \delta g_{k-2}, \dots, \delta g_0) | d_0(g_{k-1}) \delta g_{k-2}+ Td_0 (\delta g_{k-1})g_{k-2}=0, \delta g_{k-3} = \dots = \delta g_0= 0\} \\
    \ker T\bar{d}^k_2 = & \{ (\delta g_{k-1}, \delta g_{k-2}, \dots, \delta g_0) | Td_1(g_{k-1})=0,   Td_0 (\delta g_{k-2})g_{k-3} + d_0(g_{k-2}) \delta g_{k-3}=0, \delta g_{k-4} = \dots = \delta g_0= 0\} \\
    \dots \\
    \ker T\bar{d}^k_i = & \{ (\delta g_{k-1}, \delta g_{k-2}, \dots, \delta g_0) | Td_{i-1}(\delta g_{k-1})=0, Td_{i-2} (\delta g_{k-2})=0, \dots,  Td_0(\delta g_{k-i}) g_{k-i-1} + d_0(g_{k-i}) \delta g_{k-i-1} = 0,\\& \delta g_{k-i-2} = \dots = \delta g_0= 0\} \\
    \dots \\
    \ker T\bar{d}^k_k = & \{ (\delta g_{k-1}, \delta g_{k-2}, \dots, \delta g_0) | Td_{k-1}(\delta g_{k-1})=0, Td_{k-2} (\delta g_{k-2})=0, \dots,  Td_1(\delta g_1)=0 \}, 
\end{split}    
\end{equation}
where $\delta g_l$ denotes a tangent vector in $T_{g_l} G_l$ for all $l$. 
We denote the horn projection of $\bar{W}G_\bullet$ by $\bar{p}^k_j$. Then, to calculate $\ker T\bar{p}^k_0 =\cap_{i=1}^k \ker T\bar{d}_j$, we basically need to solve all the equations in the description of sets in \eqref{eq:ker-d-i}. It can be seen that $\delta g_{k-1} \in \cap_{j=1}^{k-1} \ker Td^{k-1}_j = \mathfrak{n}_{k-1}$, and $\delta g_{k-3} = \dots = \delta g_0=0$. Moreover,  $\delta g_{k-3} = \dots = \delta g_0=0$ makes all the equations proposed on $\delta g_{k-3}, \dots, \delta g_0$ trivially satisfied. From $\ker T\bar{d}_1$,  we observe that $\delta g_{k-2}$ is dertermined by $\delta g_{k-1}$ in the following way, 
\begin{equation} \label{eq:g-k-2}
    \delta g_{k-2} = d_0(g_{k-1})^{-1} Td_0 (\delta g_{k-1}) g_{k-2}.
\end{equation}Now we verify that \eqref{eq:g-k-2} implies all the equations involving $\delta g_{k-2}$ in \eqref{eq:ker-d-i}. Notice that $\delta g_{k-3} = \dots = \delta g_0=0$, thus the equations involving $\delta g_{k-2}$ reduce to
\begin{equation}
    Td_{i}(\delta g_{k-2}) = 0, \quad 0\le i \le k-2. 
\end{equation}
Notice that $d_i$'s are group morphisms, thus
\[
\begin{split}
  Td_i \Big( d_0(g_{k-1})^{-1} \cdot ( Td_0 (\delta g_{k-1}) ) \cdot g_{k-2} \Big)  = d_i(d_0(g_{k-1})^{-1}) \cdot Td_i (Td_0 (\delta g_{k-1}) )  \cdot d_i(g_{k-2})=0 . 
\end{split}
\] The last step uses the fact that $Td_i (Td_0 (\delta g_{k-1}) ) = Td_0 Td_{i+1} (\delta g_{k-1}) =0$ because $\delta g_{k-1} \in T\ker d_j$ for $j\ge 1$. Thus, $\ker T\bar{p}^k_0 = \mathfrak{n}_{k-1}$, which is simply a vector space. Therefore the tangent object of $\bar{W}G $ is \[\mathfrak n_0[1] \oplus \mathfrak n_1[2] \oplus \dots \oplus \mathfrak n_k [k+1] \oplus \dots .\]
Thus, with Theorem \ref{thm:splitgrad},  in the case of $\bar{W}G_\bullet$,  we reproduce the result calculated in \cite{Jurco}. 
\end{example}

For more examples, we also refer to \cite[Section8.4]{li:thesis} and \cite[Section13]{severa:diff}.

\bibliographystyle{alpha}
\bibliography{bibliography}

\end{document}